\documentclass{siamart190516}
\usepackage{amsfonts}
\usepackage{amsmath}
\usepackage{mathrsfs}
\usepackage{stmaryrd}
\usepackage{float}
\usepackage{nicefrac}

\title{Linearisation of the Travel Time Functional in Porous Media Flows}

\author{Paul Houston\thanks{
School of Mathematical Sciences, University of Nottingham,
University Park, Nottingham NG7 2RD, UK ({\tt
Paul.Houston@nottingham.ac.uk}, {\tt
Connor.Rourke@nottingham.ac.uk}, {\tt
KG.vanderZee@nottingham.ac.uk}). The research by KvdZ was supported by the EPSRC under Grants EP/T005157/1 and EP/W010011/1.}
\and 
Connor J. Rourke$^{*}$
\and 
Kristoffer G. van der Zee$^{*}$
}

\headers{Linearising the travel time functional}{Paul Houston, Connor J. Rourke, and Kristoffer G. van der Zee}

\begin{document}

\maketitle

\begin{abstract}
The travel time functional measures the time taken for a particle trajectory to travel from a given initial position to the boundary of the domain. Such evaluation is paramount in the post-closure safety assessment of deep geological storage facilities for radioactive waste where leaked, non--sorbing, solutes can be transported to the surface of the site by the surrounding groundwater. The accurate simulation of this transport can be attained using standard dual-weighted-residual techniques to derive goal-oriented \textit{a posteriori} error bounds. This work provides a key aspect in obtaining a suitable error estimate for the travel time functional: the evaluation of its Gâteaux derivative. A mixed finite element method is implemented to approximate Darcy's equations and numerical experiments are presented to test the performance of the proposed error estimator. In particular, we consider a test case inspired by the Sellafield site located in Cumbria, in the UK. 
\end{abstract}



 
\section{Introduction}
Over the last few decades, control of the discretisation error generated by the numerical approximation of partial differential equations (PDEs) has witnessed significant advances due to contributions in \textit{a posteriori} errror analysis and the use of adaptive mesh refinement techniques. Such algorithms aim to save computational resources by refining only a certain subset of elements, making up part of the underlying mesh, that contribute most to the error in some sense. In particular, we refer to the early works \cite{AINSWORTH19971, ivo2001finite, bangerth2003adaptive}, and the references cited therein.

Typically, in applications we are not concerned with pointwise accuracy of the numerical solution of PDEs themselves, but rather quantities involving the solution (which we will refer to as being goal quantities, or quantities of interest); in this setting goal-oriented techniques are employed to bound the error in the given quantity of interest. Work in this area was first pioneered by \cite{Becker96afeed-back, becker_rannacher_2001} and \cite{giles_suli_2002}, which established the general framework \cite{ODEN2001735, PRUDHOMME1999313} of the dual, or adjoint, weighted-residual method (DWR). When the quantity of interest is represented by a nonlinear functional, a linearisation about the numerical solution is employed in order for the problem to become tractable and computable; hence, the nonlinear functional must be differentiated. Solving a discrete version of this linearised adjoint problem allows for an estimate of the discretisation error induced by the quantity of interest, which may be localised further to drive adaptive refinement algorithms. Unweighted, residual-based estimates can be derived based on employing certain stability estimates \cite{eriksson_estep_hansbo_johnson_1995}, but this results in meshes independent of the choice of quantity of interest. The DWR approach has been applied to a vast number of different applications including the Poisson problem \cite{Becker96afeed-back}, nonlinear hyperbolic conservation laws \cite{hartmann_houston}, fluid-structure interaction problems \cite{VANDERZEE20112738}, application to Boltzmann-type equations \cite{hoitinga_van_brum}, as well as criticality problems in neutron transport applications \cite{NeutronTransport}.

In this paper, our motivation is in the post-closure safety assessment of facilities intended for use as deep geological storage of high-level radioactive waste \cite{cliffe_collis_houston, NDA, NIREX, MCKEOWN1999231}. Here, we are solely interested in the time-of-flight for a non-sorbing solute (which has leaked from the repository) to make its way to the surface, or boundary, of the domain; this time is represented by the (nonlinear) travel time functional. Previously, work undertaken in \cite{cliffe_collis_houston} employed goal-oriented \textit{a posteriori} error estimation for this functional, relying on a finite-difference  approximation of its Gâteaux derivative. 

The work presented in this article derives an exact expression for the Gâteaux derivative of the travel time functional, based on employing a backwards-in-time initial-value-problem (IVP) considered adjoint to the trajectory of the leaked solute. The use of such linearisation allows for an easy implementation of the adjoint problem required for the goal-oriented error estimation of the travel time functional. In comparison with the previous approximate linearisation, in the case of a lowest--order approximation for the driving velocity field, there is now no need for time--stepping techniques to evaluate the derivative of the travel time functional, which are often slow and computationally expensive.

Before we proceed, we first introduce the travel time functional for generic velocity fields; in addition a preliminary version of the main result of this work is presented: the Gâteaux derivative of the travel time functional for continuous velocity fields. Next we briefly discuss some of the literature relating to Darcy's equations as a model for groundwater flow, other potential models that could be used for more realistic simulations, and the \textit{a posteriori} error analysis that has been developed within these areas. Finally, we outline the contents of the rest of this article.

%
%
\subsection{The Travel Time Functional}\label{prelims}
Within this section, we define the travel time functional for generic velocity fields and address briefly the difficulties involved with its linearisation. To this end, consider an open and bounded Lipschitz domain $\Omega\subset\mathbb{R}^d$, $d = 2,3$, with polygonal boundary $\partial\Omega = \Gamma$, and the semi-infinite time interval $\mathcal{I} = [0, \infty)$. Let us suppose we have a generic velocity field $\mathbf{u} = \mathbf{u}(\mathbf{x}, t) : \overline{\Omega}\times \mathcal{I}\rightarrow\mathbb{R}^d$. For a user-defined initial position $\mathbf{x}_0\in\Omega$, the particle trajectory $\mathbf{X} \equiv \mathbf{X_u}$, due to $\mathbf{u}$, is given by the solution of the following IVP:
\begin{align*}
	 \displaystyle\frac{d\mathbf{X}}{dt}(t) & = \mathbf{u}(\mathbf{X}(t), t) \;\;\;\; \forall t\in \mathcal{I},\\
	 \mathbf{X}(0) & = \mathbf{x}_0.
\end{align*}
The so-called travel time of the velocity field $\mathbf{u}$, which is defined to be the time-of-flight of the particle trajectory $\mathbf{X_u}$ from its initial position $\mathbf{x}_0$ to, if ever, its first exit point out of the domain $\Omega$. Thereby, the functional $T(\mathbf{u}; \mathbf{x}_0)$ is defined by
\begin{equation}\label{TT}
	T(\mathbf{u}; \mathbf{x}_0) = \inf\{t\in \mathcal{I} : \mathbf{X_u}(t)\not\in\Omega\}.
\end{equation}
Alternatively, we can write this in the equivalent form:
\begin{equation*}
	T(\mathbf{u}; \mathbf{x}_0) = \int_{P(\mathbf{u}; \mathbf{x}_0)}\frac{dt}{\Vert\mathbf{u}\Vert_2},
\end{equation*}
where $\Vert\cdot\Vert_2$ denotes the standard Euclidean $2-$norm and $P(\mathbf{u}; \mathbf{x}_0)$ is the curve traced by the particle trajectory from its initial position to the first boundary contact:
\begin{equation*}
	P(\mathbf{u}; \mathbf{x}_0) = \{\mathbf{X_u}(t)\in\overline{\Omega} : t\in [0, T(\mathbf{u}; \mathbf{x}_0)]\}.
\end{equation*}
The integral version of the functional clearly highlights the difficulty concerning the demonstration of its differentiability. Indeed, the nonlinearity occurs within the integrand and the curve in which the integral is taken over depends itself on the velocity field. The travel time functional cannot clearly be globally continuous and therefore not globally Fréchet differentiable. We shall see, however, that it is possible to evaluate its Gâteaux derivative (Theorem \ref{mainresult}). The regularity of the functional itself will not be addressed within this work.

Additionally, evaluating the travel time functional itself involves the computation of the velocity streamlines, or particle trajectories $\mathbf{X_u}(t)$. Within this work, we follow the techniques outlined in \cite{KAASSCHIETER1995277} for streamline computation; furthermore, a streamfunction
approach can indeed be employed when the considered fluid flow approximations are divergence-free
\cite{MATRINGE2006992}, and it is even possible for high-order velocity approximations, when also divergence-free, to
have accurate streamline tracing \cite{juanes_matringe}.
\subsection{Linearisation in the Continuous Case}
A preliminary result for the linearisation of the travel time functional involves assuming that the velocity field $\mathbf{u}$ satisfying the underlying flow problem is continuous on $\Omega$. When this is the case, then the Gâteaux derivative of the travel time functional may be evaluated and computed as an integral, in time, weighted by a variable $\mathbf{Z}$ which may be considered as being \textit{adjoint} to the particle trajectory $\mathbf{X_u}$. The below theorem presents such a preliminary version of the main result of this paper. Here, for a sufficiently smooth functional ${\cal Q}:V\rightarrow {\mathbb R}$, we use the notation ${\cal Q}'[w](\cdot)$ to denote the G\^{a}teaux derivative of ${\cal Q}(\cdot)$ evaluated at some $w$ in $V$, where $V$ is some suitably chosen function space.
\begin{theorem}\label{preMainResult}
	Suppose that the velocity field $\mathbf{u}(\mathbf{x}, t)$ is continuous on $\Omega$. Let $\mathbf{n} = \mathbf{n}(\mathbf{x})$ be the unit outward normal vector to $\Gamma$. Assume $\Gamma$ is locally flat at the exit point $\mathbf{X}_{\mathbf{u}}(T(\mathbf{u}; \mathbf{x}_0))$, and that the particle trajectory does not exit the domain parallel to the boundary. Let $\mathbf{Z}$ solve the IVP:
	\begin{align*}
		-\frac{d\mathbf{Z}}{dt}(t) - [\nabla\mathbf{u}(\mathbf{X}(t), t)]^\top\mathbf{Z}(t) & = \mathbf{0}\;\;\;\;\forall t\in[0, T(\mathbf{u}; \mathbf{x}_0)),\\
		\mathbf{Z}(T(\mathbf{u}; \mathbf{x}_0)) & = -\frac{\mathbf{n}}{\mathbf{u}(\mathbf{X}(T(\mathbf{u}; \mathbf{x}_0)), T(\mathbf{u}; \mathbf{x}_0))\cdot\mathbf{n}}.
	\end{align*}
Then, the G\^{a}teaux derivative of the travel time functional may be evaluated as
\begin{equation*}
	T'[\mathbf{u}](\mathbf{v}) = \int_0^{T(\mathbf{u}; \mathbf{x}_0)}\mathbf{Z}(t)\cdot\mathbf{v}(\mathbf{X}(t), t)\,dt.
\end{equation*}
\end{theorem}
The above result can be used to evaluate the derivative required for the implementation of DWR {\em a posteriori} error estimators, where here the velocity field $\mathbf{u}$ is replaced with its discrete approximation $\mathbf{u}_h$. However, such approximations are usually obtained via finite element methods, and the continuity of $\mathbf{u}_h$ at element interfaces is not always guaranteed. In this case, Theorem \ref{preMainResult} must be generalised to allow for such discontinuity; this is addressed as part of Section \ref{MainResultSec}, where Theorem \ref{mainresult} is derived without such a  continuity assumption. Moreover, Theorem \ref{mainresult} presents a more general result in which Theorem \ref{preMainResult} may be recovered easily by setting the resulting jump terms equal to zero.
%
%
\subsection{Related Literature}
Groundwater flow, governed by Darcy’s equations, represents a viable simplified model for the fluid flow \cite{MCKEOWN1999231, cliffe_collis_houston} and will be exploited within this paper. It is assumed that whilst the surrounding rocks may not be saturated while the repository is being built, that they will eventually become saturated in its operational lifetime; thus, it is sufficient that in a post-closure assessment we can consider saturated conditions, and therefore use the time independent Darcy’s equations as our model, rather than the usual Richards equations for capillary flow \cite[p. 3]{collis}. Of course, within this context and in many others, there are more sophisticated models, cf. \cite{Pal, ZHANG2016396, Popov, Martin_Jaffre_Roberts, BOON_SIAM, BOON_springer, FUMAGALLI2021110205, poly_methods, Milne, Berre} and the references cited therein, where large-scale structures and complex topographical features, such as fracture networks or vugs and caves, are considered as parts of the domain. The solution-based \textit{a posteriori} error estimation for these more sophisticated models may be found in, for example, the articles \cite{CHEN2014502, chen_sal, chen_sun, WILLIAMSON2019266, Hecht, Varela, MGHAZLI2019163} and the references cited therein.

An energy norm based approach can also be found in \cite{CAO1999681}, where adaptive mesh refinement is employed  to accurately compute streamlines via a streamfunction approach. More generally, the goal-oriented error estimation for linear functionals of Darcy’s equations can be found in \cite{MOZOLEVSKI2015127} which employs equilibrated-flux techniques in order to achieve a guaranteed bound. Furthermore, \cite{MALLIK2020112367} extends this work to bound higher-order terms to demonstrate that the \textit{a posteriori} bounds are asymptotically exact, as well as taking into account the error induced by inexact solvers.

For a set of slightly different homogenised problems, \cite{Carraro} presents the goal-oriented error estimation for general quantities of interest. We also point out the existing literature for goal-adaptivity in the context of contaminant transport, presented in the articles \cite{Bengzon, LARSONgoal}, but which differs slightly from the work presented here. For the numerical experiments presented in Section \ref{numexp}, for example, following \cite{MFEM}, we employ a mixed finite element method using the Brezzi-Douglas-Marini (BDM) elements. These elements, introduced originally in \cite{Brezzi}, ensure $H$(div)-conformity in order to retain physical results in the streamline computation: that is, ensuring the continuity of the normal traces of velocity fields across element interfaces. 

The original solution-based \textit{a posteriori} error analysis for Darcy's equations, employing Raviart--Thomas elements, was undertaken by Braess and Verfürth in \cite{Braess}; we also refer to \cite{Barrios,BARRIOS2015909} which consider augmented, stabilised versions of Darcy’s equations, whose original $L^2$-bound analysis was given in the article \cite{Larson}. Moreover, there is a vast literature for the \textit{a posteriori} error analysis for Darcy’s equations in a a variety of contexts. For example, \cite{Bernadi_Orfi, ORFI20192833} presents the analysis for time-dependent Darcy flow; \cite{de_pietro} uses the finite volume method for two-phase Darcy flow; and \cite{BARRIOS_Bustinza} uses an augmented discontinuous Galerkin method. For the (residual) norm-based \textit{a posteriori} error analysis for Darcy’s equations, and mixed finite element methods in general, we refer to the articles \cite{Voh1,Voh2} by Vohralík, and the references cited therein. In \cite{Verfurth}, similar to \cite{Carst}, residual-based \textit{a posteriori} error bounds are derived by considering a Helmholtz decomposition in order to overcome the need for a saturation assumption previously assumed by \cite{Braess}. Moreover, in \cite{AMANBEK2020112884} an enhanced velocity mixed finite element method is used instead.

Lastly, problems modelled by Darcy’s equations often lend themselves for investigation in the realm of uncertainty quantification; more specifically, in real-life there is uncertainty regarding the properties of the sub-surface rock making up the domain. While not the focus of this work, we refer to \cite{collis}, and the references cited therein, where substantial work has been undertaken in a random setting.
%
%
\subsection{Outline of the Paper}
In Section \ref{modelsec} we introduce Darcy's equations for a simple model of saturated groundwater flow and their classical mixed formulation. Section \ref{MFEMsec} presents the numerical approximation of Darcy's equations via the mixed finite element method. The DWR method is presented in Section \ref{DWRsec}; here, an \textit{a posteriori} error estimate is stated and localised into element-wise indicators. Section~\ref{MainResultSec} represents the main contribution of this paper which is presented for piecewise discontinuous velocity fields. The remainder of this section proves the main linearisation result, given by Theorem \ref{mainresult}, for the travel time functional. Applying the linearisation result to groundwater flow and Darcy's equations is addressed in Section \ref{DarcyContextSec}, and the following Section \ref{ImplementationSec} provides some brief implementation details when the velocity field under consideration is piecewise linear. Three numerical experiments are conducted in Section \ref{numexp}: two simple, academic-style examples aim to build confidence in the proposed \textit{a posteriori} error estimate, while the last one adaptively simulates the leakage of radioactive waste within a domain inspired by the (albeit greatly simplified) Sellafield site, located in Cumbria, in the UK. This final, physically motivated example, matches the experiment conducted in \cite{cliffe_collis_houston} but uses the new linearisation result instead. Lastly, some concluding remarks are discussed in Section \ref{ConclusionsSec}.


\section{Darcy Flow, FE Approximation, and \textit{A Posteriori} Error Estimation}
\subsection{The Model for Groundwater Flow}\label{modelsec}
For illustrative purposes, a Darcy flow model is adopted in this paper in order to demonstrate the main Gâteaux derivative result (Theorem \ref{mainresult}) in the context of goal-oriented adaptivity. 
To this end, Darcy's equations are given by the following system of first-order PDEs, whereby we seek the \textit{Darcy velocity} $\mathbf{u}$ and \textit{hydraulic head (or pressure)} $p$ such that:
\begin{alignat}{2}
\label{DL}\mathbf{K}^{-1}\mathbf{u} + \nabla p & = \mathbf{0}\;\;\;\; && \forall\mathbf{x}\in\Omega,\\
\label{CM}\nabla\cdot\mathbf{u} & = f\;\;\;\; && \forall\mathbf{x}\in\Omega,\\
\label{Dbc} p & = g_D\;\;\;\; && \forall\mathbf{x}\in\partial\Omega_D,\\
\label{Nbc}\mathbf{u}\cdot\mathbf{n} & = 0\;\;\;\; && \forall\mathbf{x}\in\partial\Omega_N.
\end{alignat}
Here, $\Omega\subset\mathbb{R}^d$, $d=2,3$, is an open and bounded domain with polygonal boundary $\partial\Omega$, partitioned into so-called Dirichlet and Neumann parts ${\partial\Omega} = \overline{\partial\Omega}_D\cup\overline{\partial\Omega}_N$; the unit outward normal vector to the boundary is denoted by $\mathbf{n}$. Furthermore, $f\in L^2(\Omega)$ is a source/sink term and $g_D\in H^{\frac{1}{2}}(\partial\Omega_D)$ is Dirichlet boundary data for the pressure. Such regularity assumptions allow for the existence of a unique weak solution to Darcy's equations, discussed very briefly in Section \ref{formulationsec}. Lastly, the matrix $\mathbf{K}(\mathbf{x})\in\mathbb{R}^{d\times d}$ represents the hydraulic conductivity of the surrounding rock in the groundwater model; it is given by
$
\mathbf{K} := \nicefrac{\rho g}{\mu}\mathbf{k},
$
where $\rho$ is the density of water, $g$ is the acceleration due to gravity, $\mu$ is the viscosity of water, and $\mathbf{k}$ is the permeability of the surrounding rock. It is assumed that the eigenvalues of $\mathbf{K}$, $\lambda_{\pm}$ ($0 < \lambda_- \leq \lambda_+)$ satisfy
\begin{equation}\label{conductivityassumption}
\lambda_-\vert\mathbf{y}\vert^2 \leq \mathbf{y}^\top\mathbf{K}\mathbf{y} \leq \lambda_+\vert\mathbf{y}\vert^2 \;\;\;\; \forall\mathbf{x}\in\Omega\;\;\;\;\forall\mathbf{y}\in\mathbb{R}^d.
\end{equation}
In particular, the condition (\ref{conductivityassumption}) implies that $\mathbf{K}$ is invertible.


\subsubsection{Weak Formulation}\label{formulationsec}
Firstly, we introduce the following function spaces:
\begin{align*}
H(\text{div}, \Omega) & := \{\mathbf{v}\in [L^2(\Omega)]^d : \nabla\cdot\mathbf{v}\in L^2(\Omega)\},\\
H_{0, D}^1(\Omega) & := \{\psi\in H^1(\Omega) : \psi\vert_{\partial\Omega_D} = 0\},\\
H_{0, N}(\text{div}, \Omega) & := \{\mathbf{v}\in H(\text{div}, \Omega) : \langle\mathbf{v}\cdot\mathbf{n}, \psi\rangle_{\partial\Omega} = 0\;\;\forall\psi\in H_{0, D}^1(\Omega)\}.
\end{align*}
The space $H_{0, N}(\text{div}, \Omega)$ is a subspace of $H(\text{div}, \Omega)$ with vanishing normal-trace on the Neumann part of the boundary $\partial\Omega_N$. The duality pairing between $H^{-\frac{1}{2}}(\partial\Omega)$ and $H^{\frac{1}{2}}(\partial\Omega)$ is denoted by $\langle\cdot,\cdot\rangle_{\partial\Omega}$ and is given by the following Green's formula.
\begin{proposition}\label{IBP}
For $\mathbf{v}\in H(\textnormal{div}, \Omega)$,
\begin{equation*}
\langle\mathbf{v}\cdot\mathbf{n}, \psi\rangle_{\partial\Omega} = \int_\Omega \mathbf{v}\cdot\nabla\psi + \int_\Omega \psi\nabla\cdot\mathbf{v}\;\;\;\;\forall\psi\in H^1(\Omega).
\end{equation*}
\end{proposition}
By multiplying (\ref{DL}) by a test function $\mathbf{v}\in H_{0, N}(\text{div}, \Omega)$ and (\ref{CM}) by a test function $q\in L^2(\Omega)$, and applying Proposition \ref{IBP} to the latter, we arrive at the saddle-point problem: find $(\mathbf{u}, p)\in \mathbf{H} := H_{0, N}(\text{div}, \Omega)\times L^2(\Omega)$ such that
\begin{alignat}{2}
\label{SP1}& a(\mathbf{u}, \mathbf{v}) + b(\mathbf{v}, p) && = G(\mathbf{v})\;\;\;\;\forall\mathbf{v}\in H_{0, N}(\text{div}, \Omega),\\
\label{SP2}& b(\mathbf{u}, q) && = F(q)\;\;\;\;\forall q\in L^2(\Omega).
\end{alignat}
The bilinear forms are given by
$a(\mathbf{u}, \mathbf{v}) := \int_\Omega \mathbf{K}^{-1}\mathbf{u}\cdot\mathbf{v}$, 
$b(\mathbf{v}, p) := -\int_\Omega p\nabla\cdot\mathbf{v}$,
and the linear functionals are defined as
$G(\mathbf{v}) := -\langle\mathbf{v}\cdot\mathbf{n}, g_D\rangle_{\partial\Omega}$, 
$F(q) := -\int_\Omega f q$.
For simplicity of presentation, we rewrite (\ref{SP1})--(\ref{SP2}) in the following compact manner: 
find $(\mathbf{u}, p)\in \mathbf{H}$ such that
\begin{equation}
\label{SP3}\mathscr{A}((\mathbf{u}, p), (\mathbf{v}, q)) = \mathscr{L}((\mathbf{v}, q))
~~~\forall (\mathbf{v}, q)\in \mathbf{H},
\end{equation}
where
\begin{align}
\label{bilin}\mathscr{A}((\mathbf{u}, p), (\mathbf{v}, q)) & := a(\mathbf{u}, \mathbf{v}) + b(\mathbf{u}, q) + b(\mathbf{v}, p),\\
\label{lin}\mathscr{L}((\mathbf{v}, q)) & := G(\mathbf{v}) + F(q).
\end{align}
Such a weak formulation admits a unique solution $(\mathbf{u}, p)\in \mathbf{H}$ according to standard theory (see \cite{MFEM}, for example). That is, since the functionals $G$ and $F$ are clearly continuous; the pair of solution-spaces satisfy the well known \textit{BNB, or inf-sup, compatibility condition}
\begin{equation*}
0 < \beta := \inf_{0\neq\varphi\in L^2(\Omega)}\sup_{\mathbf{0}\neq\mathbf{v}\in H_{0, N}(\text{div}, \Omega)}\frac{b(\mathbf{v}, \varphi)}{\Vert\mathbf{v}\Vert_{H(\text{div}, \Omega)}\Vert\varphi\Vert_{L^2(\Omega)}},
\end{equation*}
(as a result of the divergence operator $\mathfrak{B} : H_{0, N}(\text{div}, \Omega)\rightarrow L^2(\Omega)$ $(\mathbf{w}\mapsto\nabla\cdot\mathbf{w})$ being surjective); and the bilinear form $a(\cdot, \cdot)$ being coercive on the kernel of the divergence operator $\mathfrak{B}$. 
Indeed, the surjectivity of $\mathfrak{B}$ follows immediately from the application of the \textit{Lax--Milgram Lemma} to a standard Poisson problem, giving the unique existence of a $\varphi\in H^1(\Omega)$ such that
\begin{alignat*}{2}
-\Delta\varphi & = q && \;\;\;\; \forall\mathbf{x}\in\Omega,\\
\varphi  = 0  \;\;\;\; \forall\mathbf{x}\in\partial\Omega_D, ~~~
\nabla\varphi\cdot\mathbf{n} & = 0 && \;\;\;\; \forall\mathbf{x}\in\partial\Omega_N,
\end{alignat*}
for any $q\in L^2(\Omega)$; $\varphi$ admits the function $\mathbf{w} = -\nabla\varphi\in H_{0, N}(\textnormal{div}, \Omega)$ with $\nabla\cdot\mathbf{w} = q$.


\subsection{Mixed Finite Element Approximation}\label{MFEMsec}
The numerical approximation of Darcy's equations employed in this paper will be based on a mixed finite element method. To this end, let $\mathscr{T}_h$ be a shape-regular simplicial partition of $\overline\Omega$ with $h$ the mesh-size parameter. We use the terminology \emph{face} to refer to a $(d-1)$-dimensional simplicial facet which forms part of the boundary of an element $\kappa\in\mathscr{T}_h$. 
Consider the finite-dimensional subspaces $\mathbf{V}_h\subset H_{0, N}(\text{div}, \Omega)$ and $\Pi_h\subset L^2(\Omega)$. To achieve such $H(\text{div}, \Omega)$-conformity is paramount; indeed, such approximations will have continuous normal-traces across element faces (for example, see \cite{MFEM}), allowing for the computation of physical streamlines, vital to real-life applications. Conversely, nodal-based elements should not be implemented since they often result in unphysical streamlines, as well as there being a lack of mass conservation at an elemental level \cite{Cordes}. Typically, such conformity is achieved by utilising the well known \textit{Raviart--Thomas} (RT) or \textit{Brezzi--Douglas--Marini} (BDM) finite elements. For the pressure space $\Pi_h$ we employ discontinuous piecewise-polynomial functions. However, we stress that any approximation spaces can be used as long as they are $H(\text{div}, \Omega)$ and $L^2(\Omega)$ conforming, respectively, and are a stable pair in the \textit{inf--sup} sense.
Hence, the discrete problem is: find $(\mathbf{u}_h, p_h)\in \mathbf{H}_h:=\mathbf{V}_h\times \Pi_h$ such that
\begin{equation}
\label{DiscDarcy}\mathscr{A}((\mathbf{u}_h, p_h), (\mathbf{v}_h, q_h)) = \mathscr{L}((\mathbf{v}_h, q_h))
~~~\forall (\mathbf{v}_h, q_h)\in \mathbf{H}_h.
\end{equation}



\subsection{Goal--Oriented Error Estimation}\label{DWRsec}
In this section we briefly present the general DWR theory for the \textsl{a posteriori} 
error estimation for a general nonlinear functional ${\mathcal Q}:{\mathbf H} \rightarrow {\mathbb R}$ for the flow problem \eqref{SP3}; for simplicity of 
presentation, here the underlying PDE problem is linear, though we stress that the proceeding analysis 
naturally generalises to the nonlinear setting.

To this end, given \eqref{SP3} and its corresponding finite element approximation defined by \eqref{DiscDarcy},
we define the error in the quantity of interest ${\mathcal Q}(\mathbf{u},p)$, by
\begin{equation}\label{goalerror}
 {\mathcal E}^{\mathcal Q}_h := {\mathcal Q}(\mathbf{u},p) - {\mathcal Q}(\mathbf{u}_h,p_h).
\end{equation}
To estimate this quantity we introduce the following sequence of \textit{adjoint or dual} problems, relative to the variational problem (\ref{SP3}), with respect to the functional ${\mathcal Q}$:

\textbf{Adjoint problem I:} find $(\mathbf{z}, r)\in\mathbf{H}$ such that
\begin{equation}\label{DWRformal}
\mathscr{A}((\mathbf{v}, q), (\mathbf{z}, r)) = \overline{{\mathcal Q}}((\mathbf{u}, p), (\mathbf{u}_h, p_h); (\mathbf{v}, q))\;\;\;\;\forall(\mathbf{v}, q)\in\mathbf{H},
\end{equation}
where the mean-value linearisation of ${\mathcal Q}(\cdot)$, evaluated at $\zeta\in\mathbf{H}$, is defined as 
\begin{equation}\label{mvlDef}
	\overline{{\mathcal Q}}((\mathbf{u}, p), (\mathbf{u}_h, p_h); \zeta) := \int_0^1 {\mathcal Q}'[\vartheta(\mathbf{u}, p) + (1 - \vartheta)(\mathbf{u}_h, p_h)](\zeta)\,d\vartheta.
\end{equation}

\textbf{Adjoint problem II:} find $(\mathbf{z}_\star, r_\star)\in\mathbf{H}$ such that
\begin{equation}\label{DWRlin}
\mathscr{A}((\mathbf{v}, q), (\mathbf{z}_\star, r_\star)) = {\mathcal Q}'[(\mathbf{u}_h, p_h)]((\mathbf{v}, q))\;\;\;\;\forall(\mathbf{v}, q)\in\mathbf{H}.
\end{equation}

\textbf{Discrete adjoint problem II:} find $(\mathbf{z}_h, r_h)\in\mathscr{W}_h$ such that
\begin{equation}\label{DWRdisclin}
\mathscr{A}((\mathbf{v}_h, q_h), (\mathbf{z}_h, r_h)) = {\mathcal Q}'[(\mathbf{u}_h, p_h)]((\mathbf{v}_h, q_h))\;\;\;\;\forall(\mathbf{v}_h, q_h)\in\mathscr{W}_h.
\end{equation}
Here, the finite-dimensional space $\mathscr{W}_h$ can be any space such that $\mathscr{W}_h\subset\mathbf{H}$ but so that $\mathscr{W}_h\not\subset\mathbf{H}_h$, for reasons relating to Galerkin orthogonality that we shall see later. If hierarchical bases are used within the finite-element method, then a popular choice is to have $\mathscr{W}_h$ defined on the same mesh $\mathscr{T}_h$ as $\mathbf{H}_h$, but employ higher-order polynomials. We also see already here the need to be able to evaluate the Gâteaux derivative of the nonlinear functional representing the quantity of interest, since it appears in both of the adjoint problems (\ref{DWRlin}) and (\ref{DWRdisclin}).

Defining the residual by
\begin{equation}\label{residual}
\mathscr{R}_h(\mathbf{v}, q) := \mathscr{L}((\mathbf{v}, q)) - \mathscr{A}((\mathbf{u}_h, p_h), (\mathbf{v}, q)),
\end{equation}
we have, by employing standard arguments, the following error representation formula.

\begin{proposition}[Error Representation]
Let $(\mathbf{u}, p)$ denote the solution of the primal problem (\ref{SP3}), $(\mathbf{u}_h, p_h)$ solve the discrete, primal problem (\ref{DiscDarcy}) and $(\mathbf{z}, r)$ be the solution of the adjoint problem (\ref{DWRformal}). Then, the following equality holds
\begin{equation}
 \mathcal{E}^{\mathcal Q}_h = \mathscr{R}_h(\mathbf{z} - \mathbf{z}_I, r - r_I) \label{ERimpoved}
\end{equation}
for all $(\mathbf{z}_I, r_I)\in\mathbf{H}_h$.
\end{proposition}
In particular, (\ref{ERimpoved}) is relevant for localising an estimate of the error representation, in order to potentially drive mesh adaptivity. Of course, (\ref{ERimpoved}) is not computable since the formal adjoint solutions $(\mathbf{z}, r)$ are not, in general, computable themselves. We must instead use the approximate linearised adjoint problem, and its discretisation, in order to approximate the error (\ref{goalerror}). 

To this end, we can see easily that, for all $(\mathbf{z}_I, r_I)\in\mathbf{H}_h$, the residual may be decomposed into the three parts
\begin{align*}
 \mathcal{E}^{\mathcal Q}_h 
						  & = \mathscr{R}_h(\mathbf{z} - \mathbf{z}_\star, r - r_\star) + \mathscr{R}_h(\mathbf{z}_\star - \mathbf{z}_h, r_\star - r_h) + \mathscr{R}_h(\mathbf{z}_h - \mathbf{z}_I, r_h - r_I).
\end{align*}
The first term $\mathscr{R}_h(\mathbf{z} - \mathbf{z}_\star, r - r_\star)$ represents the error induced by the approximate linearisation of the formal adjoint problem; the second term $\mathscr{R}_h(\mathbf{z}_\star - \mathbf{z}_h, r_\star - r_h)$ represents the error induced by discretising the approximate linearised adjoint problem. The last term, $\mathscr{R}_h(\mathbf{z}_h - \mathbf{z}_I, r_h - r_I)$ is most useful since it is \textit{computable}. If we assume that the other, non-computable, residuals converge to zero with an asymptotic rate \textit{faster} than this latter term, we can simply estimate the error in the quantity of interest with the computable part directly by
\begin{equation}\label{ErrorEstimate}
 \mathcal{E}_h^{\mathcal Q} \approx \mathscr{R}_h(\mathbf{z}_h - \mathbf{z}_I, r_h - r_I).
\end{equation}
Typically, the functions $\mathbf{z}_I$ and $r_I$ are chosen to be interpolants: projecting the discrete linearised adjoint solutions $\mathbf{z}_h$ and $r_h$ from $\mathscr{W}_h$ into $\mathbf{H}_h$. We stress that the presence of these interpolants are essential to ensure that the \textit{double} rate of convergence expected in optimal goal-oriented adaptive regimes is retained when local elementwise error indicators are defined based on \eqref{ErrorEstimate}, cf. below.

Under mesh refinement, whether it be uniform or adaptive, the estimate (\ref{ErrorEstimate}) converges to the true error if the \textit{effectivity index}
$\theta_h := \mathcal{E}_h^{\mathcal Q} /\mathscr{R}_h(\mathbf{z}_h - \mathbf{z}_I, r_h - r_I) \rightarrow 1$
as the mesh is refined. Section \ref{numexp} showcases numerical evidence of this behaviour for both simple and more complex examples, under uniform and adaptive refinement.


\subsubsection{Estimate Localisation for Darcy's Equations}
In this section we localise the error estimate (\ref{ErrorEstimate}) into element-based indicators on the mesh $\mathscr{T}_h$, based on the usual, integration-by-parts approach. 
To this end, writing the right-hand side of (\ref{ErrorEstimate}) as a sum over the mesh $\mathscr{T}_h$, we get
\begin{align}
  \mathcal{E}_h^{\mathcal Q} \approx \sum_{\kappa\in\mathscr{T}_h}\Big( & -\langle(\mathbf{z}_h - \mathbf{z}_I)\cdot\mathbf{n}_\kappa, g_D\rangle_{\partial\kappa\cap\partial\Omega_D} - \int_\kappa (r_h - r_I)f\nonumber\\
&-\int_\kappa\mathbf{K}^{-1}\mathbf{u}_h\cdot(\mathbf{z}_h - \mathbf{z}_I) + \int_\kappa p_h\nabla\cdot(\mathbf{z}_h - \mathbf{z}_I) + \int_\kappa (r_h - r_I)\nabla\cdot\mathbf{u}_h\Big), \label{sumovereles2}
\end{align}
where $\mathbf{n}_\kappa$ denotes the unit outward normal vector to element $\kappa\in\mathscr{T}_h$.
Employing the Green's formula stated in Proposition \ref{IBP}, we see that in particular
\begin{equation*}
\int_\kappa p_h\nabla\cdot(\mathbf{z}_h - \mathbf{z}_I) = -\int_\kappa(\mathbf{z}_h - \mathbf{z}_I)\cdot\nabla p_h + \langle(\mathbf{z}_h - \mathbf{z}_I)\cdot\mathbf{n}_\kappa, p_h\rangle_{\partial\kappa}.
\end{equation*}
Therefore, summing over the elements in the mesh, gives 
\begin{align}
\nonumber \sum_{\kappa\in\mathscr{T}_h}\int_\kappa p_h\nabla\cdot(\mathbf{z}_h - \mathbf{z}_I) = \sum_{\kappa\in\mathscr{T}_h}\Big(& -\int_\kappa(\mathbf{z}_h - \mathbf{z}_I)\cdot\nabla p_h + \frac{1}{2}\langle(\mathbf{z}_h - \mathbf{z}_I)\cdot\mathbf{n}_\kappa, \llbracket p_h\rrbracket\rangle_{\partial\kappa\setminus\partial\Omega}\\
& \label{SlotTerm2}+ \langle(\mathbf{z}_h - \mathbf{z}_I)\cdot\mathbf{n}_\kappa, p_h\rangle_{\partial\kappa\cap\partial\Omega_D}\Big),
\end{align}
where $\llbracket \cdot \rrbracket$ denotes the jump operator across an element face.
Inserting (\ref{SlotTerm2}) into (\ref{sumovereles2}) gives the following result.
\begin{theorem}\label{DarcyLocal}
Under the foregoing notation, we have the (approximate) a posteriori error estimate
\begin{equation*}
\vert \mathcal{E}_h^{\mathcal Q}\vert \approx \bigg\vert\sum_{\kappa\in\mathscr{T}_h}\eta_\kappa\bigg\vert \leq \sum_{\kappa\in\mathscr{T}_h}\vert\eta_\kappa\vert
\end{equation*}
where the element indicator $\eta_\kappa$ is split into the four contributions 
\begin{equation*}
\eta_\kappa \equiv \eta_\kappa^{BC} + \eta_\kappa^{DL} + \eta_\kappa^{CM} + \eta_\kappa^{PR},
\end{equation*} 
each given by: 
\begin{align}
\label{I1} \eta_\kappa^{BC} & = \langle(\mathbf{z}_h - \mathbf{z}_I)\cdot\mathbf{n}_\kappa, p_h - g_D\rangle_{\partial\kappa\cap\partial\Omega_D},\\
\label{I2} \eta_\kappa^{DL} & = -\int_\kappa(\mathbf{K}^{-1}\mathbf{u}_h + \nabla p_h)\cdot(\mathbf{z}_h - \mathbf{z}_I),\\
\label{I3} \eta_\kappa^{CM} & = \int_\kappa(r_h - r_I)(\nabla\cdot\mathbf{u}_h - f),\\
\label{I4} \eta_\kappa^{PR} & = \frac{1}{2}\langle(\mathbf{z}_h - \mathbf{z}_I)\cdot\mathbf{n}_\kappa, \llbracket p_h\rrbracket\rangle_{\partial\kappa\setminus\partial\Omega}.
\end{align}
\end{theorem}
Each of the indicator contributions (\ref{I1})--(\ref{I4}) are \textit{adjoint-weighted} and may be interpreted as the following: $\eta_\kappa^{BC}$ measures how well the boundary condition (\ref{Dbc}) is satisfied; $\eta_\kappa^{DL}$ measures how well Darcy's Law (\ref{DL}) is satisfied; $\eta_\kappa^{CM}$ measures how well the conservation of mass equation (\ref{CM}) is satisfied; and finally, $\eta_\kappa^{PR}$ is a measure of the interior pressure residual across element interfaces.


\section{Linearising the Travel Time Functional}\label{MainResultSec}
Recalling the discussion presented in Section \ref{prelims}, we emphasise that the main result (i.e. evaluating the Gâteaux derivative of the travel time functional) is independent of where the velocity field $\mathbf{u}$ has come from; for now we are concerned only about the continuity of $\mathbf{u}$. Indeed, computing an approximation to the travel time functional via an approximation of the velocity field $\mathbf{u}$ may or may not lead to a continuous velocity field; this depends on the fluid model and the type of approximation that is employed.

More explicitly: suppose our problem was not in groundwater flow and the disposal of radioactive waste, but instead that we are interested in $T(\mathbf{u};\mathbf{x}_0)$ where $\mathbf{u}$ is a flow governed by Stokes equations. In this situation, typically vector-valued $H^1$-conforming elements are employed (cf. \cite{BS}), on some mesh $\mathscr{T}_h$, to obtain an approximation (at least in two spatial dimensions) $\mathbf{u}_h$ that is continuous across the element interfaces. Here, Theorem \ref{preMainResult} can be applied to evaluate the derivative $T'[\mathbf{u}_h](\cdot)$ (to, for example, drive an adaptive mesh refinement algorithm). However, in the context of this work, an $H(\text{div})$-conforming approximation of a flow governed by Darcy's equations is used and as such, this conformity does not guarantee continuity of the velocity field across element interfaces. Thereby, in the following discussion we derive a more general result stated in Theorem \ref{mainresult}.


\subsection{Linearisation in the Discontinuous Case}
Given the domain $\Omega\subset\mathbb{R}^d$, $d=2,3$, denote by $\mathcal{I}$ the semi-infinite time interval $[0, \infty)$. Furthermore, suppose we have the possibly time-dependent velocity field
$
\mathbf{v} : (\mathbb{R}^d\times\mathcal{I})\rightarrow\mathbb{R}^d.
$
The particle trajectory of the velocity field, $\mathbf{X}_\mathbf{v} : \mathcal{I}\rightarrow\mathbb{R}^d$, satisfies the IVP:
\begin{equation}\label{traj}
\begin{cases}\frac{d\mathbf{X}_{\mathbf{v}}}{dt} = \mathbf{v}(\mathbf{X}_\mathbf{v}, t) & \;\;\;\;\forall t\in\mathcal{I},\\
             \mathbf{X}_\mathbf{v}(0) = \mathbf{x}_0,\end{cases}
\end{equation}
where the initial position $\mathbf{x}_0\in\Omega$.

The main result is stated below in Theorem \ref{mainresult}, which provides the evaluation of the Gâteaux derivative $T'[\mathbf{v}](\cdot)$, of the travel time functional $T(\cdot)$.

\begin{theorem}\label{mainresult}
Let $\mathbf{n} = \mathbf{n}(\mathbf{x})$ be the unit outward normal vector to the boundary $\partial\Omega$. Assume firstly that $\partial\Omega$ is locally flat at the exit point $\mathbf{X}(T_\mathbf{v})$, so that the unit outward normal vector $\mathbf{n} = \mathbf{n}(\mathbf{X}(T_\mathbf{v}))$ is unique. Assume also that the particle trajectory does not exit $\Omega$ parallel to the boundary, so that $\mathbf{v}(\mathbf{X}(T_\mathbf{v}), T_\mathbf{v})\cdot\mathbf{n}(\mathbf{X}(T_\mathbf{v})) \neq 0$. Suppose that $\mathscr{T}_h$ is a simplicial partition of $\Omega$ and that $\mathbf{v}$ is discontinuous across the faces $\{\mathcal{F}_i\}$ that intersect the path $t\mapsto\mathbf{X}(t)$, defined by (\ref{traj}) at the times $\{t_i = t_{i,\mathbf{v}}\}$. Lastly assume that the particle trajectory does not exit any element in $\mathscr{T}_h$ parallel to its boundary, or through the boundary of one of the element faces, except possibly at the end, given the previous assumption about local flatness. With the above notation described, let $\mathbf{Z} : [0, T_\mathbf{v}]\rightarrow\mathbb{R}^d$ be the solution to the adjoint, or dual (linearised-adjoint, backward-in-time) IVP:
\begin{equation}\label{adjointIVP}
\begin{cases}\mathcal{L}_\mathbf{v}^*(\mathbf{Z}(t)) \equiv -\frac{d\mathbf{Z}}{dt} - [\nabla\mathbf{v}(\mathbf{X}(t), t)]^\top\mathbf{Z} = \mathbf{0} & \;\;\;\;\forall t\in[0, T_\mathbf{v})\setminus\{t_{i,\mathbf{v}}\},\\ \\
\mathbf{Z}(T_\mathbf{v}) = -\frac{\mathbf{n}(\mathbf{X}(T_\mathbf{v}))}{\mathbf{v}(\mathbf{X}(T_\mathbf{v}), T_\mathbf{v})\cdot\mathbf{n}(\mathbf{X}(T_\mathbf{v}))},\\ \\
\llbracket\mathbf{Z}(t_{i,\mathbf{v}})\rrbracket = -\frac{\mathbf{Z}(t_{i,\mathbf{v}}^+)\cdot\llbracket\mathbf{v}(t_{i,\mathbf{v}})\rrbracket\mathbf{n}_i^-}{\mathbf{v}(\mathbf{X}(t_{i,\mathbf{v}}^-), t_{i,\mathbf{v}}^-)\cdot\mathbf{n}_i^-} & \;\;\;\;\forall i,\end{cases}
\end{equation}
where $\mathbf{n}_i^-$ is the unit outward normal vector to the faces $\{\mathcal{F}_i\}$, pointing in the same direction as the particle trajectory $\mathbf{X_v}(t)$ at the time of intersection $t = t_i$, and where $\llbracket\mathbf{Z}(t_{i,\mathbf{v}})\rrbracket = \mathbf{Z}(t_{i,\mathbf{v}}^+) - \mathbf{Z}(t_{i,\mathbf{v}}^-)$ and $\llbracket\mathbf{v}(t_{i,\mathbf{v}})\rrbracket = \mathbf{v}(\mathbf{X}(t_{i, \mathbf{v}}^+), t_{i, \mathbf{v}}^+) - \mathbf{v}(\mathbf{X}(t_{i, \mathbf{v}}^-), t_{i, \mathbf{v}}^-)$ denote jump operators. Then, the Gâteaux derivative of $T(\cdot)$, evaluated at $\mathbf{v}$, is given by
\begin{equation*}
T'[\mathbf{v}](\mathbf{w}) = \int_0^{T_\mathbf{v}}\mathbf{Z}(t)\cdot\mathbf{w}(\mathbf{X}(t), t)\,dt.
\end{equation*}
\end{theorem}

The plus/minus notation refers to the times after/before, respectively, the trajectory $\mathbf{X_u}$ intersects the element interface, forwards in time.
We may also index $\mathbf{Z_v} \equiv \mathbf{Z}$ to indicate that $\mathbf{Z_v}$ solves the IVP (\ref{adjointIVP}) induced by the velocity field $\mathbf{v}$. Also, we note that if the velocity field driving the trajectory is in fact continuous across the element interfaces, then the jump terms vanish and Theorem~\ref{preMainResult} is recovered.

We now proceed to prove Theorem \ref{mainresult}. To this end, we require two lemmas which are given below. Firstly, consider the so-called trajectory derivative, corresponding to the change in the particle path as a result of a change in velocity:
\begin{equation*}
\mathbf{X}'\equiv\partial_\mathbf{v}\mathbf{X}_\mathbf{v}[\mathbf{w}] := \lim_{\varepsilon\rightarrow 0^+}\frac{\mathbf{X}_{\mathbf{v} + \varepsilon\mathbf{w}} - \mathbf{X}_\mathbf{v}}{\varepsilon},
\end{equation*}
recalling the notation that $\mathbf{X}_\mathbf{v}$ is the trajectory induced by the velocity field $\mathbf{v}$.

\begin{lemma}\label{lemma1}
Let $\mathbf{v}$ be as before, discontinuous across the faces $\{\mathcal{F}_i\}$ intersecting the path $t\mapsto\mathbf{X}_\mathbf{v}(t)$ at the times $\{t_i = t_{i, \mathbf{v}}\}$. Then, the trajectory derivative $\mathbf{X}' : \mathcal{I}\rightarrow \mathbb{R}^d$ satisfies the IVP:
\begin{equation}\label{derivjump}
\begin{cases}\mathcal{L}_\mathbf{v}(\mathbf{X}'(t))\equiv \frac{d\mathbf{X}'}{dt} - \nabla\mathbf{v}(\mathbf{X}_\mathbf{v}(t), t)\mathbf{X}' = \mathbf{w}(\mathbf{X}_\mathbf{v}(t), t) & \;\;\;\;\forall t\in\mathcal{I}\setminus\{t_i\},\\
\mathbf{X}'(0) = \mathbf{0},\\
\llbracket\mathbf{X}'(t_i)\rrbracket = -\llbracket\mathbf{v}(t_i)\rrbracket t_i' & \;\;\;\;\forall i,  \end{cases}
\end{equation}
where 
\begin{equation}\label{ti_label}
t_i' = -\frac{\mathbf{X}'(t_i^-)\cdot\mathbf{n}_i^-}{\mathbf{v}(\mathbf{X}_\mathbf{v}(t_i^-), t_i^-)\cdot\mathbf{n}_i^-}.
\end{equation}
\end{lemma}

\begin{proof}
The time derivative of $\mathbf{X}'$ is given by
\begin{equation*}
\frac{d\mathbf{X}'}{dt} = \frac{d}{dt}\lim_{\varepsilon\rightarrow 0^+}\frac{\mathbf{X}_{\mathbf{v} + \varepsilon\mathbf{w}} - \mathbf{X}_\mathbf{v}}{\varepsilon} = \lim_{\varepsilon\rightarrow 0^+}\frac{(\mathbf{v} + \varepsilon\mathbf{w})(\mathbf{X}_{\mathbf{v} + \varepsilon\mathbf{w}}, t) - \mathbf{v}(\mathbf{X}_\mathbf{v}(t), t)}{\varepsilon},
\end{equation*}
where we recall the pathline equations the trajectories satisfy. Thus,
\begin{align*}
\frac{d\mathbf{X}'}{dt} & = \lim_{\varepsilon\rightarrow 0^+}\frac{(\mathbf{v} + \varepsilon\mathbf{w})(\mathbf{X}_{\mathbf{v} + \varepsilon\mathbf{w}}, t) - \mathbf{v}(\mathbf{X}_\mathbf{v}(t), t)}{\varepsilon}\\
& = \lim_{\varepsilon\rightarrow 0^+}\frac{\mathbf{v}(\mathbf{X}_{\mathbf{v} + \varepsilon\mathbf{w}}(t), t) - \mathbf{v}(\mathbf{X}_\mathbf{v}(t), t)}{\varepsilon} + \mathbf{w}(\mathbf{X}_\mathbf{v}(t), t)\\
& = \lim_{\varepsilon\rightarrow 0^+}\frac{\mathbf{v}(\mathbf{X}_\mathbf{v}(t) + \varepsilon\mathbf{X}'(t) + o(\varepsilon), t) - \mathbf{v}(\mathbf{X}_\mathbf{v}(t), t)}{\varepsilon} + \mathbf{w}(\mathbf{X}_\mathbf{v}(t), t)\\
& = \lim_{\varepsilon\rightarrow 0^+}\frac{[\nabla\mathbf{v}(\mathbf{X}_\mathbf{v}(t), t)]\epsilon\mathbf{X}'(t) + o(\varepsilon)}{\varepsilon} + \mathbf{w}(\mathbf{X}_\mathbf{v}(t), t)\\
& = [\nabla\mathbf{v}(\mathbf{X}_\mathbf{v}(t), t)]\mathbf{X}'(t) + \mathbf{w}(\mathbf{X}_\mathbf{v}(t), t),
\end{align*}
i.e., for all $t\in\mathcal{I}\setminus\{t_i\}$ (so that $\nabla\mathbf{v}(\mathbf{X}_\mathbf{v}(t), t)$ exists away from the discontinuities),
\begin{equation*}
\frac{d\mathbf{X}'}{dt} - [\nabla\mathbf{v}(\mathbf{X}_\mathbf{v}(t), t)]\mathbf{X}'(t) = \mathbf{w}(\mathbf{X}_\mathbf{v}(t), t).
\end{equation*}
The initial condition follows easily as
\begin{equation*}
\mathbf{X}'(0) = \lim_{\varepsilon\rightarrow 0^+}\frac{\mathbf{X}_{\mathbf{v} + \varepsilon\mathbf{w}}(0) - \mathbf{X}_\mathbf{v}(0)}{\varepsilon} = \lim_{\varepsilon\rightarrow 0^+}\frac{\mathbf{x}_0 - \mathbf{x}_0}{\varepsilon} = \mathbf{0}.
\end{equation*}
Although the velocity $\mathbf{v}$ has discontinuities, we still require that the trajectory $\mathbf{X}_\mathbf{v}$ is continuous. Hence, we have the coupling conditions between the two maps:
\begin{equation*}
(\mathbf{v}\mapsto\mathbf{X}_\mathbf{v}(t_i^+)) = (\mathbf{v}\mapsto\mathbf{X}_\mathbf{v}(t_i^-))\;\;\;\;\forall i.
\end{equation*}
Taking the Gâteaux derivative of each side (i.e., $(d/d\varepsilon)(\cdot)(\mathbf{v} + \varepsilon\mathbf{w})$, as $\varepsilon\rightarrow 0$) gives
\begin{equation*}
\mathbf{X}'(t_i^+) + \frac{d\mathbf{X}(t_i^+)}{dt}t_i' = \mathbf{X}'(t_i^-) + \frac{d\mathbf{X}(t_i^-)}{dt}t_i'\;\;\;\;\forall i.
\end{equation*}
Thus,
\begin{equation*}
\mathbf{X}'(t_i^+) + \mathbf{v}(\mathbf{X}_\mathbf{v}(t_i^+), t_i^+)t_i' = \mathbf{X}'(t_i^-) + \mathbf{v}(\mathbf{X}_\mathbf{v}(t_i^-), t_i^-)t_i'\;\;\;\;\forall i;
\end{equation*}
rearranging gives
\begin{equation*}
\llbracket\mathbf{X}'(t_i)\rrbracket = -\llbracket\mathbf{v}(t_i)\rrbracket t_i'.
\end{equation*}
The expression for $t_i' \equiv \partial_\mathbf{v}t_{i,\mathbf{v}}(\mathbf{w})$, given by (\ref{ti_label}), follows similarly to the proof given for the following Lemma \ref{final}.
\end{proof}
We note as well that a variational approach can be used instead to prove Lemma \ref{lemma1}.
For use in Lemma \ref{final}, consider the change in exit-time, or time-of-flight, due to a change in the velocity, given by
\begin{equation*}
T'\equiv \partial_\mathbf{v}T_\mathbf{v}(\mathbf{w}) := \lim_{\varepsilon\rightarrow 0^+}\frac{T_{\mathbf{v} + \varepsilon\mathbf{w}} - T_\mathbf{v}}{\varepsilon}.
\end{equation*}

\begin{lemma}\label{final}
The derivative $\mathbf{X}'(T_\mathbf{v})$ satisfies
\begin{equation*}
\mathbf{X}'(T_\mathbf{v})\cdot\mathbf{n} = -T'\mathbf{v}(\mathbf{X}_\mathbf{v}(T_\mathbf{v}), T_\mathbf{v})\cdot\mathbf{n},
\end{equation*}
with $\mathbf{n} \equiv \mathbf{n}(\mathbf{X}_\mathbf{v}(T_\mathbf{v}))$.
\end{lemma}

\begin{proof}
If we assume that the boundary $\partial\Omega$ is locally flat at the exit-point $\mathbf{X}(T_\mathbf{v})$, then this means that for sufficiently small $\varepsilon$ we have
\begin{equation*}
(\mathbf{X}_{\mathbf{v} + \varepsilon\mathbf{w}}(T_\mathbf{v}) - \mathbf{X}_{\mathbf{v}}(T_\mathbf{v}))\cdot\mathbf{n} = (\mathbf{X}_{\mathbf{v} + \varepsilon\mathbf{w}}(T_{\mathbf{v}}) - \mathbf{X}_{\mathbf{v} + \varepsilon\mathbf{w}}(T_{\mathbf{v} + \varepsilon\mathbf{w}}))\cdot\mathbf{n}.
\end{equation*}
Hence,
\begin{align*}
\mathbf{X}'(T_\mathbf{v})\cdot\mathbf{n} & = \lim_{\varepsilon\rightarrow 0^+}\frac{\mathbf{X}_{\mathbf{v} + \varepsilon\mathbf{w}}(T_\mathbf{v}) - \mathbf{X}_\mathbf{v}(T_\mathbf{v})}{\varepsilon}\cdot\mathbf{n}\\
& = \lim_{\varepsilon\rightarrow 0^+}\frac{\mathbf{X}_{\mathbf{v} + \varepsilon\mathbf{w}}(T_{\mathbf{v}}) - \mathbf{X}_{\mathbf{v} + \varepsilon\mathbf{w}}(T_{\mathbf{v} + \varepsilon\mathbf{w}})}{\varepsilon}\cdot\mathbf{n}\\
& = \lim_{\varepsilon\rightarrow 0^+}\frac{\mathbf{X}_{\mathbf{v} + \varepsilon\mathbf{w}}(T_{\mathbf{v}}) - \mathbf{X}_{\mathbf{v} + \varepsilon\mathbf{w}}(T_\mathbf{v} + \varepsilon T' + o(\varepsilon))}{\varepsilon}\cdot\mathbf{n}\\
& = \lim_{\varepsilon\rightarrow 0^+}\frac{-\frac{d\mathbf{X}_{\mathbf{v} + \varepsilon\mathbf{w}}}{dt}(T_\mathbf{v})(\varepsilon T' + o(\varepsilon))}{\varepsilon}\cdot\mathbf{n}\\
& = \lim_{\varepsilon\rightarrow 0^+}\frac{-(\mathbf{v} + \varepsilon\mathbf{w})(\mathbf{X}_{\mathbf{v} + \varepsilon\mathbf{w}}(T_\mathbf{v}), T_\mathbf{v})(\varepsilon T' + o(\varepsilon))}{\varepsilon}\cdot\mathbf{n}\\
& = -T'\mathbf{v}(\mathbf{X}_\mathbf{v}(T_\mathbf{v}), T_\mathbf{v}))\cdot\mathbf{n}.
\end{align*}
\end{proof}
Thus, we are now able to prove the main result. Firstly, note that
\begin{equation*}
T'[\mathbf{v}](\mathbf{w}) = \lim_{\varepsilon\rightarrow 0}\frac{T_{\mathbf{v} + \varepsilon\mathbf{w}} - T_\mathbf{v}}{\varepsilon} = T'.
\end{equation*}


\subsubsection{Proof of Theorem \ref{mainresult}}
\begin{proof}
From Lemma \ref{final} and (\ref{adjointIVP}) we have
\begin{equation*}
T' = -\frac{\mathbf{X}'(T_\mathbf{v})\cdot\mathbf{n}}{\mathbf{v}(\mathbf{X}_\mathbf{v}(T_\mathbf{v}), T_\mathbf{v})\cdot\mathbf{n}} = \mathbf{X}'(T_\mathbf{v})\cdot\mathbf{Z}(T_\mathbf{v}).
\end{equation*}
Since from (\ref{adjointIVP}) we know that $\mathcal{L}_\mathbf{v}^*(\mathbf{Z}(t)) = 0$ away from the jump times $\{t_i\}$, we have
\begin{equation*}
T' \equiv \mathbf{X}'(T_\mathbf{v})\cdot\mathbf{Z}(T_\mathbf{v}) = \mathbf{X}'(T_\mathbf{v})\cdot\mathbf{Z}(T_\mathbf{v}) + \sum_i\int_{t_{i-1}}^{t_i}\mathcal{L}_\mathbf{v}^*\mathbf{Z}(t)\cdot\mathbf{X}'(t)\,dt.
\end{equation*}
Integrating by parts reveals that
\begin{align*}
T' & \equiv \sum_i\int_{t_{i-1}}^{t_i}\mathbf{Z}(t)\cdot\mathcal{L}_\mathbf{v}X'\,dt + \sum_i(\mathbf{Z}(t_i^+)\cdot\mathbf{X}'(t_i^+) - \mathbf{Z}(t_i^-)\cdot\mathbf{X}'(t_i^-)) + \mathbf{Z}(0)\cdot\mathbf{X}'(0)\\
& = \sum_i\int_{t_{i-1}}^{t_i}\mathbf{Z}(t)\cdot\mathbf{w}(\mathbf{X}_\mathbf{v}(t), t)\,dt + \sum_i(\mathbf{Z}(t_i^+)\cdot\mathbf{X}'(t_i^+) - \mathbf{Z}(t_i^-)\cdot\mathbf{X}'(t_i^-)),
\end{align*}
since from (\ref{derivjump}) in Lemma \ref{lemma1}  we have that $\mathcal{L}_\mathbf{v}(\mathbf{X}'(t)) = \mathbf{w}(\mathbf{X}_\mathbf{v}(t), t)$ and $\mathbf{X}'(0) = \mathbf{0}$. The jump condition in (\ref{derivjump}) for $\mathbf{X}'$ can be rearranged to obtain the expression
\begin{equation*}
\mathbf{X}'(t_i^+) = \mathbf{X}'(t_i^-) + \llbracket\mathbf{v}(t_i)\rrbracket\frac{\mathbf{X}'(t_i^-)\cdot\mathbf{n}_i^-}{\mathbf{v}(\mathbf{X}_\mathbf{v}(t_i^-), t_i^-)\cdot\mathbf{n}_i^-}.
\end{equation*}
Thereby,
\begin{align*}
T' \equiv & \sum_i\int_{t_{i-1}}^{t_i}\mathbf{Z}(t)\cdot\mathbf{w}(\mathbf{X}_\mathbf{v}(t), t)\,dt\\ & + \sum_i\left(\mathbf{Z}(t_i^+)\cdot\left(\mathbf{X}'(t_i^-) + \llbracket\mathbf{v}(t_i)\rrbracket\frac{\mathbf{X}'(t_i^-)\cdot\mathbf{n}_i^-}{\mathbf{v}(\mathbf{X}_\mathbf{v}(t_i^-), t_i^-)\cdot\mathbf{n}_i^-}\right) - \mathbf{Z}(t_i^-)\cdot\mathbf{X}'(t_i^-)\right).
\end{align*}
Notice that
\begin{align*}
& \mathbf{Z}(t_i^+)\cdot\left(\mathbf{X}'(t_i^-) + \llbracket\mathbf{v}(t_i)\rrbracket\frac{\mathbf{X}'(t_i^-)\cdot\mathbf{n}_i^-}{\mathbf{v}(\mathbf{X}_\mathbf{v}(t_i^-), t_i^-)\cdot\mathbf{n}_i^-}\right) - \mathbf{Z}(t_i^-)\cdot\mathbf{X}'(t_i^-)\\ & = \left(\mathbf{Z}(t_i^+) - \mathbf{Z}(t_i^-) + \frac{\mathbf{Z}(t_i^-)\llbracket\mathbf{v}(t_i)\rrbracket}{\mathbf{v}(\mathbf{X}_\mathbf{v}(t_i^-), t_i^-)\cdot\mathbf{n}_i^-}\cdot\mathbf{n}_i^-\right)\cdot\mathbf{X}'(t_i^-)\\
& = (\llbracket\mathbf{Z}(t_i)\rrbracket - \llbracket\mathbf{Z}(t_i)\rrbracket)\cdot\mathbf{X}'(t_i^-) = \mathbf{0}.
\end{align*}
This implies that
\begin{equation*}
T' \equiv \sum_i\int_{t_{i-1}}^{t_i}\mathbf{Z}(t)\cdot\mathbf{w}(\mathbf{X}_\mathbf{v}(t), t)\,dt = \int_0^{T_\mathbf{v}}\mathbf{Z}(t)\cdot\mathbf{w}(\mathbf{X}_\mathbf{v}(t), t)\,dt,
\end{equation*}
thus completing the proof.
\end{proof}


\subsection{Application to Darcy Flow}\label{DarcyContextSec}
For a groundwater flow model governed by Darcy's equations (\ref{DL})--(\ref{Nbc}), physical (non-sorbing, non-dispersive, purely advective transport based) particle trajectories are due to a velocity field known as the transport velocity, which relates the Darcy velocity $\mathbf{u}$ and the porosity, $\phi$, of the surrounding rock via
$
\mathbf{u}_T = \nicefrac{\mathbf{u}}{\phi}.
$
Indeed, the travel time along particle trajectories driven by this velocity field are those that should be considered in the travel time functional (\ref{TT}). With $\mathbf{x}_0$ the initial burial point, our quantity of interest can be expressed either by the functionals $\mathfrak{T}(\cdot\,; \mathbf{x}_0)$ or $T(\cdot\,; \mathbf{x}_0)$, where, in particular, the former is given by
\begin{equation}\label{traveltime}
\mathfrak{T}(\mathbf{u}; \mathbf{x}_0) = T(\mathbf{u}_T; \mathbf{x}_0) = \inf\{t > 0 : \mathbf{X}_{\mathbf{u}_T}(t)\not\in\Omega\},
\end{equation}
and it is indeed the trajectory $\mathbf{X}_{\mathbf{u}_T}$ that should be considered $(\mathbf{v} \leftrightarrow \mathbf{u}_T)$ in Theorem~\ref{mainresult}, and the functional $T(\mathbf{u}_T; \mathbf{x}_0)$ should be considered in the context of the \textit{a posteriori} error estimation presented in Section \ref{DWRsec}.

Furthermore, a simple application of a generalised chain rule allows us to deduce an expression for the Gâteaux derivative of the functional $\mathfrak{T}(\cdot\,;\mathbf{x}_0)$, given by
\begin{equation}\label{chain}
\mathfrak{T}'[\mathbf{v}](\mathbf{w}) = T'[\mathbf{v}_T](\mathbf{w}_T).
\end{equation}


\subsection{Implementation Details}\label{ImplementationSec}
In this section, let $\mathbf{u}_h\in\mathbf{V}_h$ and $\mathbf{v}\in\mathbf{V}$ be generic velocity fields. For example, $\mathbf{u}_h$ could be the solution of the discrete problem (\ref{DiscDarcy}), while $\mathbf{v}$ could be a basis function of $\mathbf{W}_h\subset\mathbf{V}$, $\mathbf{W}_h\not\subset\mathbf{V}_h$, so that the derivative
\begin{equation}\label{derivative}
T'[\mathbf{u}_h](\mathbf{v}) = \int_0^{T(\mathbf{u}_h)}\mathbf{Z}(t)\cdot\mathbf{v}(\mathbf{X}_{\mathbf{u}_h}(t))\,dt
\end{equation}
is required for computing the numerical solution to the approximate linearised adjoint problem (\ref{DWRlin}). Of course, if $\mathbf{u}_h$ is the discrete Darcy velocity satisfying (\ref{DiscDarcy}) then the derivative $\mathfrak{T}'[\mathbf{u}_h](\mathbf{v})$ can be evaluated combining this section with (\ref{chain}).

For simplicity of presentation, we restrict this discussion to $d=2$, but we stress that the generalisation to $d=3$ follows directly. In this setting, we recall that  $\mathscr{T}_h$ is a shape-regular triangulation of $\overline{\Omega}$ for which $\mathbf{u}_h$ is discontinuous across the element interfaces intersected by the particle trajectory $\mathbf{X}_{\mathbf{u}_h}(t)$ at the times $\{t_i\}_{i=1}^N$; proceed with the assumptions stated in Theorem \ref{mainresult}. 
Denote by $\mathbb{T}_h = \{\kappa_i\}_{i=1}^{N}\subset\mathscr{T}_h$ the ordered list of elements intersected by the particle trajectory. Here, we allow for repetitions if the trajectory re-enters the same element, where it will appear multiple times in $\mathbb{T}_h$ with different labels. In order to obtain the adjoint variable $\mathbf{Z}_{\mathbf{u}_h} \equiv \mathbf{Z}$, we can solve the IVP (\ref{adjointIVP}) in a element-by-element manner. That is, starting from the intersection point with the boundary of $\mathbf{X}_{\mathbf{u}_h}(t)$, we trace the particle trajectory backwards through its intersected elements, and solve for $\mathbf{Z}$ on each time interval that the trajectory is residing in that element. More precisely, consider the final element $\kappa_N$. The trajectory $\mathbf{X}_{\mathbf{u}_h}(t)$ occupies this element for $t\in(t_{N-1}, t_N)$, where $t_N \equiv T(\mathbf{u}_h; \mathbf{x}_0)$ is the travel time. Restricting to this time interval, the adjoint variable $\mathbf{Z}(t)$ solves the IVP
\begin{equation*}
-\frac{d\mathbf{Z}(t)}{dt} - [\nabla\mathbf{u}_h(\mathbf{X}_{\mathbf{u}_h}(t))]^\top\mathbf{Z}(t) = \mathbf{0}.
\end{equation*}
For times $t\in(t_{N-1}, t_N)$, we have $\mathbf{X}_{\mathbf{u}_h}(t)\in\kappa_N$ and within this element $\mathbf{u}_h$ is a polynomial function. This means that together with the given final--time condition
\begin{equation*}
\mathbf{Z}(t_N) = -\frac{\mathbf{n}}{\mathbf{u}_h(\mathbf{X}(t_N))\cdot\mathbf{n}},
\end{equation*}
we can solve for $\mathbf{Z}$ within this time interval, via an exact method or using some approximate time-stepping technique for ODEs. For example, if $\mathbf{u}_h$ is a piecewise linear function on the triangulation $\mathscr{T}_h$ (e.g. a lowest order RT or BDM function) then we may solve for $\mathbf{Z}$ directly via matrix exponentials. Indeed, the gradient of such a function will be piecewise constant on the same triangulation.

In such a case, denote by $\mathbf{a} = (\alpha_x, \alpha_y)^\top$, $\mathbf{b} = (\beta_x, \beta_y)^\top$ and $\mathbf{c} = (\gamma_x, \gamma_y)^\top$ the real coefficients such that on $\kappa_i\in\mathbb{T}_h$
\begin{equation*}
\mathbf{u}_h\vert_{\kappa_i} \equiv \begin{bmatrix}\alpha_x + \beta_x x + \gamma_x y\\ \alpha_y + \beta_y x + \gamma_y y\end{bmatrix}.
\end{equation*}
Then,
$\mathbf{a} = \mathbf{u}_h\vert_{\kappa_i}(0, 0)$, 
$\mathbf{b} = \mathbf{u}_h\vert_{\kappa_i}(1, 0) - \mathbf{a}$, 
$\mathbf{c} = \mathbf{u}_h\vert_{\kappa_i}(0, 1) - \mathbf{a}$,
and the gradient of $\mathbf{u}_h$ restricted to $\kappa_i$ is given by
\begin{equation*}
\nabla\mathbf{u}_h\vert_{\kappa_i} = \begin{bmatrix} \mathbf{b} & \mathbf{c}\end{bmatrix} = \begin{bmatrix}\beta_x & \gamma_x\\ \beta_y & \gamma_y\end{bmatrix}.
\end{equation*}
Denoting by $\Upsilon_i = [\nabla\mathbf{u}_h(\mathbf{X}_{\mathbf{u}_h}(t))]^\top\vert_{\kappa_i}$ the gradient transposed for each $i$, we then have
\begin{equation}\label{AdjointSolutionFinal}
\mathbf{Z}(t) = \texttt{exp}(\Upsilon_N(t_N - t))\mathbf{Z}(t_N)\;\;\;\;\forall t\in(t_{N-1}, t_N].
\end{equation}
By putting $t = t_{N-1}$ in (\ref{AdjointSolutionFinal}), we can evaluate $\mathbf{Z}(t_{N-1}^+)$. The jump condition in (\ref{adjointIVP}) can be rearranged for the value of $\mathbf{Z}$ at this time before the particle trajectory $\mathbf{X}_{\mathbf{u}_h}(t)$ crosses into the element $\kappa_N$, forwards in time, which is given by
\begin{equation}\label{rearrange}
\mathbf{Z}(t_{N-1}^-) = \mathbf{Z}(t_{N-1}^+) + \frac{\mathbf{Z}(t_{N-1}^+)\cdot\llbracket\mathbf{u}_h(t_{N-1})\rrbracket\mathbf{n}_{N-1}}{\mathbf{u}_h(\mathbf{X}(t_{N-1}))\cdot\mathbf{n}_{N-1}}.
\end{equation}
We see that all of the terms on the right-hand-side of the equality in (\ref{rearrange}) are known (also, the orientation of the normal vector $\mathbf{n}_{N-1}$ to the element interface does not matter since it appears both in the numerator and demoninator). On the next (or previous, from the perspective of the particle trajectory) element, $\kappa_{N-1}$, we restrict to the time interval $(t_{N-2}, t_{N-1})$ and solve similarly. Now, using $\mathbf{Z}(t_{N-1}^-)$ as the final--time condition to obtain
\begin{equation*}
\mathbf{Z}(t) = \texttt{exp}(\Upsilon_{N-1}(t_{N-1} - t))\mathbf{Z}(t_{N-1}^-)\;\;\;\;\forall t\in(t_{N-2}, t_{N-1}).
\end{equation*}
One then follows this procedure for all time intervals up to and including $(0, t_1)$. In general, for a piecewise linear velocity field $\mathbf{u}_h$, we may hence write
\begin{equation}\label{AdjointSolutionLinear}
\mathbf{Z}(t) = \texttt{exp}(\Upsilon_i(t_i - t))\mathbf{Z}(t_i^-)\;\;\;\;\forall t\in(t_{i-1}, t_i).
\end{equation}
When $\mathbf{u}_h$ is, for example, piecewise polynomial with a higher degree, or some other general function, then (\ref{AdjointSolutionLinear}) does not apply since the matrices $\Upsilon_i$ will not be constant. Instead, one could employ a time--stepping technique within each time interval to solve for the adjoint solution $\mathbf{Z}(t)$; time--stepping from $\mathbf{Z}(t_i^-)$ until $\mathbf{Z}(t_{i-1}^+)$, using this to generate the next starting position $\mathbf{Z}(t_{i-1}^-)$, and so forth.

We note as well that the integral (\ref{derivative}) can be reduced to a sum of integrals over these time-intervals for which the trajectory intersects the support of the function $\mathbf{v}$. This is especially useful when $\mathbf{v}$ is, for example, a finite element basis function, which has support on only a few elements of which either all or just one might intersect the trajectory. Because of this, and the need to compute $\mathbf{Z}(t)$ in the fashion stated above, the right-hand-side vector in (\ref{DWRdisclin}) can easily be assembled by looping over these intersected elements in the same backwards fashion as described here.


\section{Numerical Examples}\label{numexp}
The purpose of this section is to utilise the linearisation result stated in Theorem \ref{mainresult} within the context of goal-oriented adaptivity. Here, Darcy's equations (\ref{DL})--(\ref{Nbc}) model the flow of groundwater as a saturated porous medium; we are interested (cf. Sections \ref{prelims}, \ref{DWRsec} and \ref{DarcyContextSec}) in the accurate estimation of the discretisation error induced by numerically approximating the travel time $\mathfrak{T}(\mathbf{u}; \mathbf{x}_0)$, for a given burial point $\mathbf{x}_0\in\Omega$. For simplicity we assume throughout this section that $d=2$.


\subsection{Approximation Spaces and Mesh Adaptivity}
Adaptive mesh refinement, and goal--oriented error estimation, will be performed for the accurate computation of the travel time functional (\ref{traveltime}) when the primal solution $(\mathbf{u}, p)\in\mathbf{H}$ to (\ref{SP3}) is approximated by the solution $(\mathbf{u}_h, p_h)\in\mathbf{H}_h$ to (\ref{DiscDarcy}). We wish to measure
\begin{equation}\label{ExpEst}
{\mathcal E}_h^\mathfrak{T} = \mathfrak{T}(\mathbf{u}; \mathbf{x}_0) - \mathfrak{T}(\mathbf{u}_h; \mathbf{x}_0) \approx \sum_{\kappa\in\mathscr{T}_h}\eta_\kappa
\end{equation}
on each of the computational meshes employed, where the indicators are those defined in Theorem \ref{DarcyLocal}. For mesh adaptivity we utilise a fixed--fraction marking strategy, with a refinement selection of $\texttt{REF} = 10\%$, together with the standard red--green, regular, refinement strategy for triangular elements.

We begin by stating the definition of the approximation space $\mathbf{H}_h$. Here, we employ the Brezzi--Douglas--Marini elements for the approximation of the Darcy velocity, and discontinuous piecewise polynomials for the approximation of the pressure (cf. Section \ref{MFEMsec}). To this end, we define the following spaces, where $\mathscr{T}_h$ is the usual shape--regular triangulation of the domain $\Omega\subset\mathbb{R}^2$:
\begin{align*}
BDM_k(\kappa) & := [\mathbb{P}_k(\kappa)]^2,\\
BDM_k(\Omega, \mathscr{T}_h) & := \{\mathbf{v}\in H(\text{div}, \Omega) : \mathbf{v}\vert_\kappa\in BDM_k(\kappa)\;\;\forall\kappa\in\mathscr{T}_h\}.
\end{align*}
Then, the approximation space $\mathbf{H}_{h, k} \equiv \mathbf{V}_{h, k}\times\Pi_{h, k}$ is defined via
\begin{align*}
\mathbf{V}_{h, k} & := \{\mathbf{v}\in BDM_{k + 1}(\Omega, \mathscr{T}_h) : (\mathbf{v}\cdot\mathbf{n})\vert_{\partial\Omega_N} = 0\},\\
\Pi_{h, k} & := \{\varphi\in L^2(\Omega) : \varphi\vert_\kappa\in\mathbb{P}_k(\kappa)\,\,\forall\kappa\in\mathscr{T}_h\}.
\end{align*}
The stability of these pairs of spaces, in the inf--sup sense, is discussed, for example, in \cite{MFEM} for any choice of $k \geq 0$.

In our examples we consider the primal and adjoint approximations $(\mathbf{u}_h, p_h)\in\mathbf{H}_{h, 0}$ and
$(\mathbf{z}_h, r_h)\in\mathbf{H}_{h, 1}$, where $(\mathbf{z}_h, r_h)$ solves the discrete linearised adjoint problem (\ref{DWRdisclin}) with functional $\mathfrak{T}(\,\cdot\,; \mathbf{x}_0)$, approximating the solutions $(\mathbf{z}, r)\in\mathbf{H}$ to the problem (\ref{DWRformal}).
We recall (cf. Section \ref{DWRsec}) the effectivity index
\begin{equation*}
\theta_h := \frac{\mathfrak{T}(\mathbf{u}; \mathbf{x}_0) - \mathfrak{T}(\mathbf{u}_h; \mathbf{x}_0)}{\sum_{\kappa\in\mathscr{T}_h}\eta_k},
\end{equation*}
which measures how well the error estimate approximates the exact travel time error.


\begin{figure}[t!]\label{SimplePath}
\centering
\includegraphics[width=0.35\columnwidth]{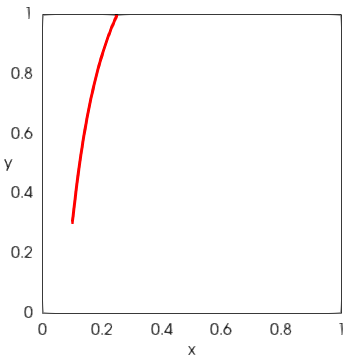}
\caption{Example I: Approximate particle trajectory on the final mesh.}
\end{figure}
\subsection{Example I: A Simple Test Case}
This first example considers a very simple problem for which we know the value of the exact travel time $\mathfrak{T}(\mathbf{u}; \mathbf{x}_0)$. The travel time is approximated on a series of uniformly refined triangulations, in order to validate the proposed error estimate (\ref{ExpEst}).
To this end, let $\Omega = (0, 1)^2$; we impose appropriate boundary conditions, so that the exact Darcy velocity is given by $\mathbf{u} = [\sin(x) \cos(y)]^\top.$ The porosity is set to be $\phi = 1$ everywhere so that the Darcy and transport velocities coincide. Furthermore, the de-coupling of the IVP for the particle trajectory $\mathbf{X_u}(t)$ means that we can evaluate exactly the travel time for some choice of $\mathbf{x}_0\in\Omega$. Selecting $\mathbf{x}_0 = (0.1,\,\,0.3)$ gives
\begin{equation*}
\mathfrak{T}(\mathbf{u}; \mathbf{x}_0) = \log\left(\frac{\tan(1) + \sec(1)}{\tan(0.3) + \sec(0.3)}\right)\approx 0.9216\dots,
\end{equation*} 
cf. Figure~\ref{SimplePath} which depicts the particle trajectory.

\begin{figure}[t!]\label{SimplePlots}
\centering
\includegraphics[width=0.4\columnwidth]{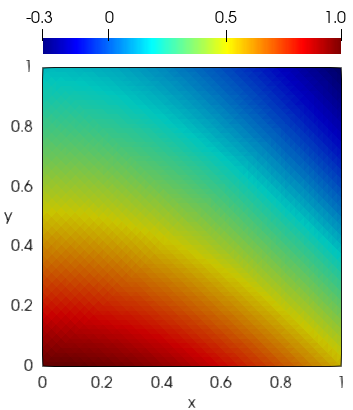}
\includegraphics[width=0.4\columnwidth]{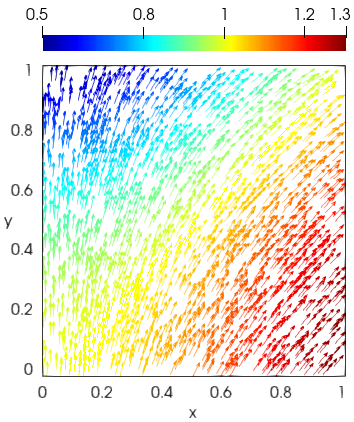}\\
\includegraphics[width=0.4\columnwidth]{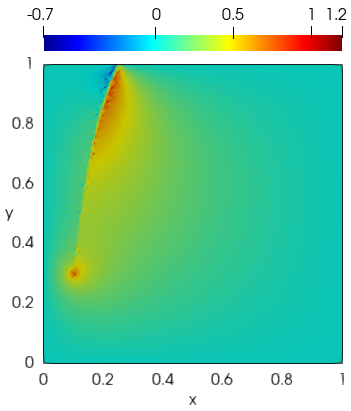}
\includegraphics[width=0.4\columnwidth]{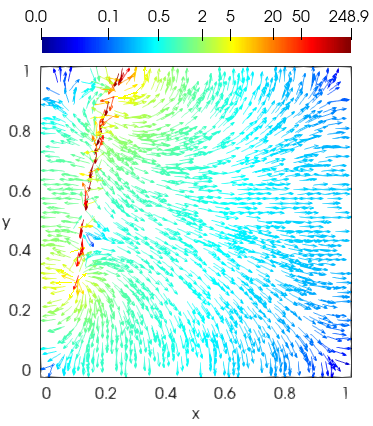}
\caption{Example I: Primal (top) and adjoint (bottom) pressure and velocity approximations on the final mesh.}
\end{figure}

\begin{table}[t!]
{\footnotesize 
  \caption{Example I: Results employing the $BDM_1$ finite element space.}\label{tab:example_I_table}
  \begin{center}
   \begin{tabular}{|c|c|c|c|} \hline
    \bf Number of DOFs & \bf Error & \bf Est. Error & \bf $\theta_h$ \\ \hline
     $20$ &    $-8.274\times 10^{-3}$ & $-8.476\times 10^{-3}  $ & $0.976$ \\ 
     $72$ &    $1.358\times 10^{-3}$ & $1.360\times 10^{-3}$ & $0.998$ \\ 
     $272$ &   $-3.155\times 10^{-5} $ & $-2.818\times 10^{-5} $ & $1.120$ \\ 
     $1056$ &  $-1.894\times 10^{-5} $ & $-1.899\times 10^{-5} $ & $0.997$ \\ 
     $4160$ &  $-2.085\times 10^{-6} $ & $-2.084\times 10^{-6} $ & $1.001$ \\ 
     $16512$ & $-9.310\times 10^{-7} $ & $-9.308\times 10^{-7} $ & $1.000$ \\ \hline
   \end{tabular}
  \end{center}
  }
\end{table}

 The results featured in Table \ref{tab:example_I_table} show the exact travel time error, the error estimate, and the resulting effectivity index on each of the uniform meshes employed for this example. Indeed, here we observe that the effectivity indices are extremely close to unity on each of the meshes, thereby demonstrating that the error estimate accurately predicts the travel time error in this simple example, even on particularly coarse meshes with less than $50$ degrees of freedom. The primal and adjoint pressure and velocity approximations on the final mesh are depicted in Figure~\ref{SimplePlots}.


\subsection{Example II: A Two--Layered Geometry}
Similar to Example I, this numerical experiment considers a simple geometry and problem set-up in order to further validate the proposed error estimate (\ref{ExpEst}) under uniform refinement. Here, the domain $\Omega$, pictured in Figure \ref{TwoLayerPath}, is defined by
$
\Omega = \{(x, y)\in(0,\,\, 1)\times (0,\,\, 1) : y + \nicefrac{x}{10} < 1\}.
$
Along the line $y = \nicefrac{1}{2}$ the domain is partitioned into the two sub-domains $\Omega_i$, $i = 1,2$, representing different types of rock. That is, the top layer consists solely of Calder Sandstone, while the bottom containes St. Bees Sandstone. To each of the sub-domains we assign a fixed, constant, permeability and porosity (cf. Example III), given by the dataset used in \cite{cliffe_collis_houston}. 

This example can be considered to be a simpler version of Example III, in which we apply the same boundary conditions. Along the top of the domain we impose atmospheric pressure, and no--flow out of the rest of the boundary. The burial point is chosen to be $\mathbf{x}_0 = (0.1,\,\, 0.1)$ and we set $f = 0$ in Darcy's equations (\ref{DL})--(\ref{Nbc}). Unlike the previous example, the exact travel time $\mathfrak{T}(\mathbf{u}; \mathbf{x}_0)$ is not known in this case; instead, we use an approximation on the final mesh.

\begin{figure}[t!]\label{TwoLayerPath}
\centering
\includegraphics[width=0.35\columnwidth]{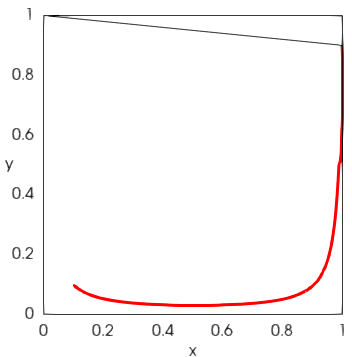}
\caption{Example II: Approximate particle trajectory on the final mesh.}
\end{figure}

\begin{figure}[t!]\label{TwoLayerPlots}
\centering
\includegraphics[width=0.4\columnwidth]{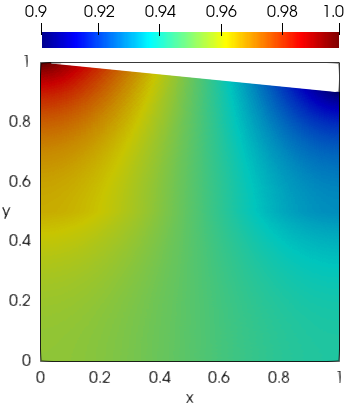}
\includegraphics[width=0.4\columnwidth]{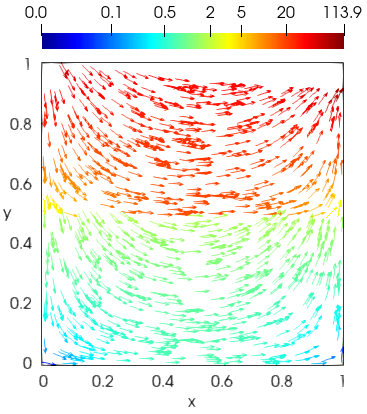}\\
\includegraphics[width=0.4\columnwidth]{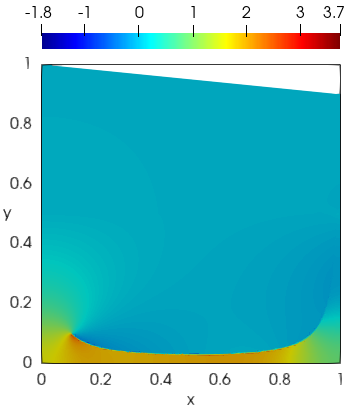}
\includegraphics[width=0.4\columnwidth]{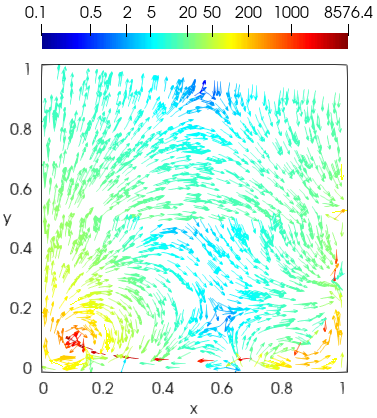}
\caption{Example II: Primal (top) and adjoint (bottom) pressure and velocity approximations on the final mesh.}
\end{figure}

  \begin{table}[t!]
{\footnotesize 
  \caption{Example II: Results employing the $BDM_1$ finite element space.}\label{tab:example_II_table}
  \begin{center}
   \begin{tabular}{|c|c|c|c|} \hline
    \bf Number of DOFs & \bf Error & \bf Est. Error & \bf $\theta_h$ \\ \hline
     $198$ &    $1.188\times 10^{-3}$ & $1.719\times 10^{-3}  $ & $0.691$ \\ 
     $764$ &    $4.773\times 10^{-4}$ & $4.534\times 10^{-4}$ & $1.053$ \\ 
     $3000$ &   $7.891\times 10^{-5} $ & $8.178\times 10^{-5} $ & $0.965$ \\ 
     $11888$ &  $1.255\times 10^{-5} $ & $1.294\times 10^{-5} $ & $0.970$ \\ 
     $47328$ &  $4.261\times 10^{-6} $ & $4.460\times 10^{-6} $ & $0.955$ \\ 
     $188864$ & $-2.694\times 10^{-7} $ & $-2.694\times 10^{-7} $ & $1.000$ \\ \hline
   \end{tabular}
  \end{center}
  }
  \end{table}

 The results presented in Table \ref{tab:example_II_table} again show that the proposed error estimate reliably predicts the size of the error, with effectivity indices close to unity on each of the meshes employed. Although it looks as if the trajectory is exiting the domain parallel to the boundary (cf. Figure \ref{TwoLayerPath}), the performance of the error estimator does not deteriorate in this setting.

 The adjoint solution approximations, pictured in Figures \ref{TwoLayerPlots}, are discontinuous along the particle trajectory. Indeed, close to $\mathbf{x}_0$ is a sink--like feature, with the adjoint velocity traveling backwards along the path to the initial point, while moving in the same direction as the path elsewhere. These may be interpretted as generalised Green's functions along the particle trajectory; in particular, the adjoint pressure looks to be bounded, while the adjoint velocity resembles more a Dirac--type measure.


\subsection{Example III: Inspired by the Sellafield Site}
In this example, the domain $\Omega$ is defined as being the union of six sub--domains $\Omega_i$, $i = 1, 2, \dots, 6$, each representing a different type of rock. Each of these layers are assumed to have a given fixed, constant, porosity $\phi$ and permeability $\mathbf{k}$ related to the hydraulic conductivity $\mathbf{K}$ (cf. Sections \ref{DarcyContextSec} and \ref{modelsec}, respectively) by
$
\mathbf{K} = \nicefrac{\rho g}{\mu}\mathbf{k},
$
where $\rho$, $g$, and $\mu$ are the density of water, acceleration due to gravity, and kinematic velocity of water, respectively; the data for each of these is taken from \cite{cliffe_collis_houston}.

We briefly mention that the domain $\Omega$ is merely inspired by the geological units found at the Sellafield site and in no way is physically representative of it; therefore, we draw no conclusions of real-life consequence within this numerical example in the context of the post-closure safety assessments of potential radioactive waste burial sites. Furthermore, this experiment merely aims to reproduce similar results previously obtained in \cite{cliffe_collis_houston} in order to verify the main linearisation result presented in Theorem \ref{mainresult}. More details concerning this problem, as well as a more complex version of this test case, can be found in \cite{cliffe_collis_houston} where the permeability per layer was considered variable, but still constant per element.

\begin{figure}[t!]\label{Sellafield}
\centering
\includegraphics[width=0.65\columnwidth]{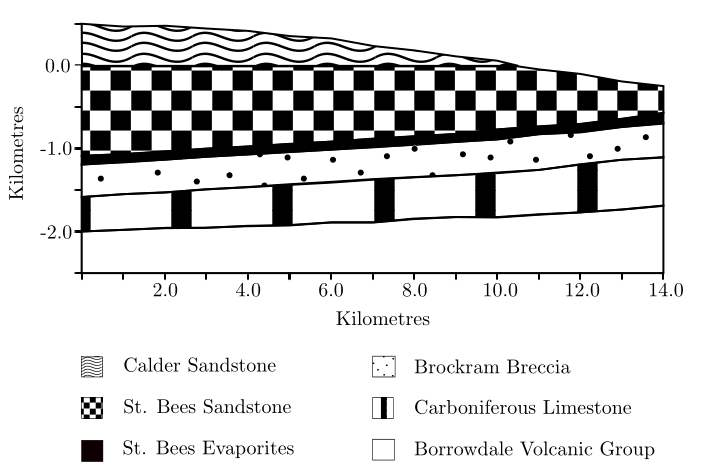}
\caption{Example III: The domain $\Omega$, inspired by Sellafield; see \cite{cliffe_collis_houston}.}
\end{figure}

\begin{figure}[t!]
\centering
\includegraphics[width=0.9\columnwidth]{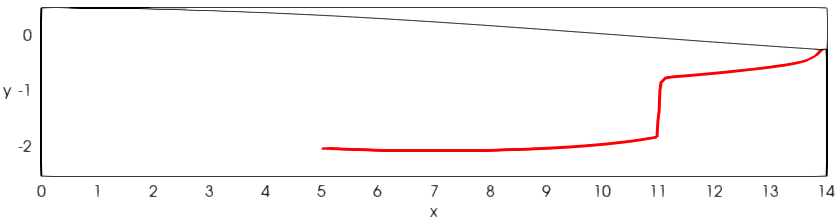}
\caption{Example III: Particle trajectory approximation on the initial mesh.}\label{SellafieldPath}
\end{figure}

 \begin{table}[t!]
{\footnotesize 
  \caption{Example III: Results employing the $BDM_1$ finite element space.}\label{tab:BDM_table}
  \begin{center}
   \begin{tabular}{|c|c|c|c|} \hline
    \bf Number of DOFs & \bf Error & \bf Est. Error & \bf Eff. Index \\ \hline
     $22871$ &               $-8.905\times 10^{-5} $ & $-5.970\times 10^{-5}  $ & $1.492$ \\ 
     $32624$ &               $-5.455\times 10^{-6}$ & $-4.421\times 10^{-6}$ & $1.234$ \\ 
     $47053$ &               $4.065\times 10^{-6} $ & $4.382\times 10^{-6} $ & $0.928$ \\ 
     $69887$ &               $-2.140\times 10^{-7} $ & $-2.206\times 10^{-7} $ & $0.970$ \\ 
     $1.0755\times 10^{5}$ & $-4.216\times 10^{-8} $ & $-4.326\times 10^{-8} $ & $0.974$ \\ 
     $1.6796\times 10^{5}$ & $-1.330\times 10^{-8} $ & $-1.468\times 10^{-8} $ & $0.906$ \\ 
     $2.6631\times 10^{5}$ & $-8.280\times 10^{-9} $ & $-8.280\times 10^{-9} $ & $1.000$ \\ \hline
   \end{tabular}
  \end{center}
  }
  \end{table}

Here, we let $\partial\Omega_D$ be the top of the domain, representing the surface of the site, and let $\partial\Omega_N$ be the remainder of the boundary, as pictured in Figure \ref{Sellafield}. We make the same assumptions as \cite{cliffe_collis_houston}:
the rock below the stratum consisting of Borrowdale Volcanic Group type is of much lower permeability than all of the other layers; there is a flow divide on the left and right edges of the domain; the pressure at the top of the domain is prescribed via $g_D = p_{\text{atm}}/\rho g + y$ where $p_{\text{atm}} \approx 1.013\times 10^5 Pa$ is atmospheric pressure; the source term $f$ is set equal to zero. The travel time path computed on the initial mesh is depicted in Figure~\ref{SellafieldPath}.

\begin{figure}[t!]
\centering
\includegraphics[width=1\columnwidth]{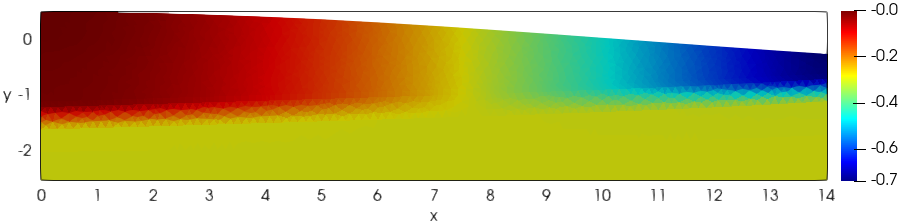}
\caption{Example III: Pressure approximation on the initial mesh.}\label{SellafieldPressure}
\end{figure}

\begin{figure}[t!]
\centering
\includegraphics[width=1\columnwidth]{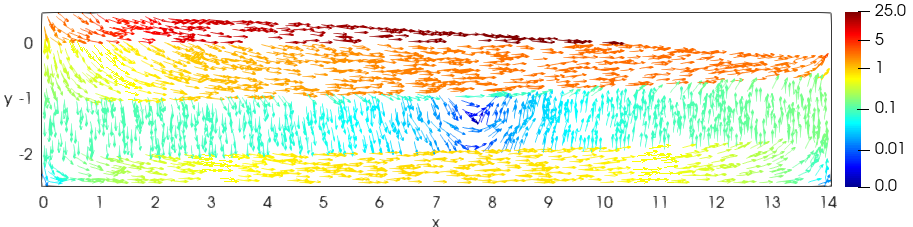}
\caption{Example III: Velocity approximation on the initial mesh.}\label{SellafieldVelocity}
\end{figure}
\begin{figure}[t!]
\centering
\includegraphics[width=1\columnwidth]{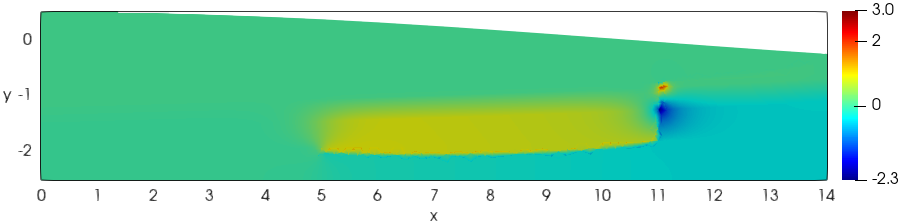}
\caption{Example III: Adjoint pressure approximation on the initial mesh.}\label{SellafieldAdjointPressure}
\end{figure}
\begin{figure}[t!]
\centering
\includegraphics[width=1\columnwidth]{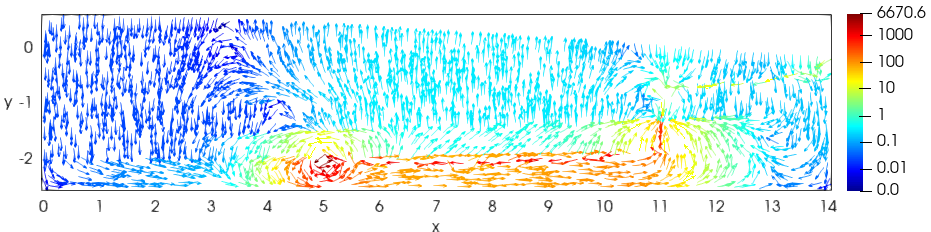}
\caption{Example III: Adjoint velocity approximation on the initial mesh.}\label{SellafieldAdjointVelocity}
\end{figure}

\begin{figure}[t!]
\centering
\includegraphics[width=0.9\columnwidth]{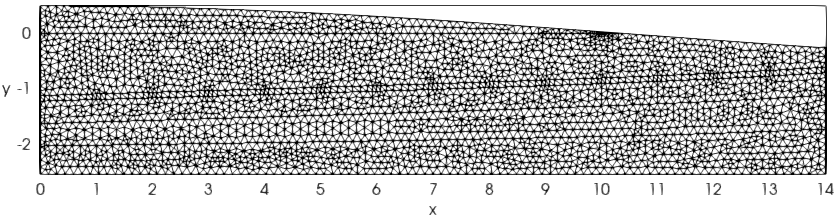}\\
\includegraphics[width=0.9\columnwidth]{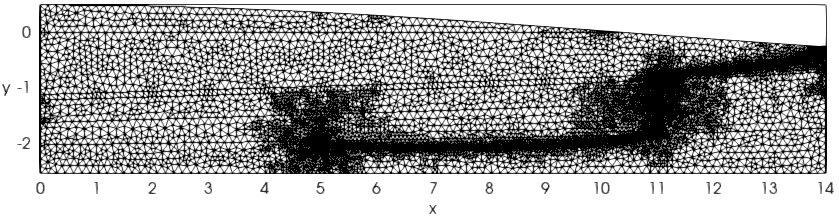}
\caption{Example III: Initial and final adaptively refined meshes.}\label{SellafieldMeshes}
\end{figure}

In Table \ref{tab:BDM_table} we present the performance of the adaptive routine when approximating the travel time functional. The exact travel time $\mathfrak{T}(\mathbf{u}; \mathbf{x}_0)$ is based on the approximation computed on the final mesh and the computed error estimator; on this basis the exact travel time is approximately $0.49\times 10^{5}$ years. We can see from these results that the effectivity indices computed on all meshes are close to unity, indicating that the approximate error estimate (\ref{ExpEst}) leads to reliable error estimation, similar to the previously undertaken work in \cite{cliffe_collis_houston}. We see that for this physically motivated example we are able to estimate the error in the travel time functional very closely.

Figures \ref{SellafieldPressure} and \ref{SellafieldVelocity} show the computed approximations $(\mathbf{u}_h, p_h)\in\mathbf{H}_{h, 0}$ on the initial mesh. Again, here we observe discontinuities in the Darcy velocity across the rock layer interfaces, with the velocities differing by orders of magnitude within each of the stratum. We also see a local stationary point in the pressure near the centre of the domain which accounts for the change in direction of the groundwater flow; indeed, in this region the flow moves upwards and thus could transport the buried nuclear waste back up to the surface of the site.

Figures \ref{SellafieldAdjointPressure} and \ref{SellafieldAdjointVelocity} plot the computed adjoint approximations $(\mathbf{z}_h, r_h)\in\mathbf{H}_{h, 1}$. As concurred by \cite{cliffe_collis_houston} we see a strong discontinuity along the direction of the trajectory $\mathbf{X}_{\mathbf{u}_{h}}$, and with both the adjoint velocity and pressure approximations vanishing away from the path. Close to the initial release point $\mathbf{x}_0$ we see what looks to be a sink--like feature in the adjoint velocity approximation, and again, in agreement with \cite{cliffe_collis_houston}, this velocity points in the same direction as the primal Darcy velocity (approximation) outside of the path, but in the opposite direction along $\mathbf{X}_{\mathbf{u}_{h}}$.

Finally, in Figure \ref{SellafieldMeshes} we show the initial mesh and the final, adaptively refined, mesh. As expected, we observe mesh refinement taking place around the initial point $\mathbf{x}_0$, at the exit point, and along the trajectory itself. There is more significant refinement (compared with the rest of the path) where the trajectory changes direction; in these regions there are sharp discontinuities in the Darcy velocity approximation, which may lead to a large discretisation error of the primal Darcy problem. Such large errors contribute greatly to the error induced in the travel time functional and as such, is targetted more for refinement when compared with the regions containing long horizontal stretches of the trajectory; typically here, the velocity (especially when confined to a single rock layer) appears to be quite smooth.


\section{Conclusions}\label{ConclusionsSec}
This work has been concerned with the numerical approximation of the travel time functional in porous media flows and the post--closure safety assessment of radioactive waste storage facilities. An expression for the Gâteaux derivative of the travel time functional has been derived, for both continuous and piecewise--continuous velocity fields, which was then utilised via the dual--weighted--residual--method for goal--oriented error estimation and mesh adaptivity. Numerical experiments considering both simple and complicated problem set--ups were considered, validating the proposed error estimate which performed extremely well, in terms of the computed effectivity indices being very close to unity on all meshes employed. The contributions of this research have built upon those in \cite{cliffe_collis_houston} where previously such an expression for the Gâteaux derivative was unavailable.

Extensions of this work may, for example, involve considering more realistic conditions in order to test the proposed error estimate. More demanding domains, such as fractured porous media or domains with inclusions such as vugs or caves, is vital to extend the results from these simple academic test cases to real--life applications. Furthermore, a closer look into the regularity of the adjoint solutions would be extremely beneficial in understanding how to improve the error estimate to derive a guaranteed bound and to better understand the expected rates of convergence in the error of the computed travel time functional. Indeed, the well--posedness of the adjoint problem still remains an open question.


\bibliographystyle{siamplain}
\bibliography{TT_References}

\end{document}